\documentclass[a4 paper,12pt]{article}
\usepackage[top=3cm, left=2.5cm, bottom=3cm, right=2.5cm]{geometry}
\usepackage{geometry}
\usepackage{color}
\usepackage{graphicx}
\usepackage{amsmath, amsthm, amssymb}
\usepackage{overpic}

\usepackage{dsfont}

\RequirePackage[colorlinks, hyperindex, plainpages=false]{hyperref} 
\usepackage[ngerman, french, english]{babel}
\usepackage[latin1]{inputenc}
\usepackage[T1]{fontenc}
\usepackage{color}
\usepackage{tikz}
\usetikzlibrary{plotmarks}
\usepackage{setspace}
\usepackage{longtable}
\usepackage{bm}
\usepackage{float}
\newcommand{\R} {\mathbb{R}}

\newtheorem{theorem}{Theorem}[section]
\newtheorem{proposition}[theorem]{Proposition}

\newtheorem{corollary}[theorem]{Corollary}

\newtheorem{lemma}[theorem]{Lemma}

\newtheorem{assumption}[theorem]{Assumption}

\newtheorem{definition}[theorem]{Definition}

\DeclareMathOperator\dist{dist}

\begin{document}

\title{Large Deviations of the Exit Measure through a Characteristic Boundary for a Poisson driven SDE.}
\author{
{Etienne Pardoux}\footnote{{Aix--Marseille Universit\'e, CNRS, Centrale Marseille, I2M, UMR 7373 13453
Marseille, France; etienne.pardoux@univ-amu.fr; brice.samegni-kepgnou@univ-amu.fr.}}
\and
{Brice Samegni-Kepgnou}\setcounter{footnote}{6}$^\ast$
}

\maketitle

\begin{abstract} 
Let $O$ the basin of attraction of the unique stable equilibrium of a dynamical system, which is the law of large numbers limit of a Poissonian SDE. We consider the law of the exit point from $O$ of that Poissonian SDE. We adapt the approach of M. Day (1990) for the same problem for a Brownian SDE. For that purpose, we will use the Large deviation for the Poissonian SDE reflected at the boundary of $O$ studied in our recent work Pardoux and Samegni (2018).   
\end{abstract}

\vskip 3mm
\noindent{\bf AMS subject classification: } 60F10, 60H10, 92C60.

\noindent{\bf Keywords: } Poisson process, Law of large numbers, Large deviation principle.
\vskip 3mm

\section{Introduction}

We consider a $d$-dimensional process of the type
\begin{equation}\label{EqPoisson1}
Z^{N}(t)=Z^{N, z}(t):= \frac{[Nz]}{N} +\frac{1}{N} \sum_{j=1}^{k} h_{j} P_{j}\Big(\int_{0}^{t} N \beta_{j}(Z^{N,z}(s)) ds \Big).
\end{equation}
Here $(P_{j})_{1\le j\le k}$ are i.i.d. standard Poisson processes. The $h_{j}\in \mathds{Z}^{d}$ denote the $k$ respective jump directions with respective jump rates $\beta_{j}(z)$ and $z\in A$ (where $A$ is the ``domain'' of the process). In fact \eqref{EqPoisson1} specifies a continuous time Markov chain with state space
\begin{equation}\label{statespace}
A^{(N)}=\Big\{z\in\mathbb{R}^d_+: (Nz_{1},\dots Nz_{d})\in\mathbb{Z}^{d}_{+},\ \sum_{i=1}^d z_i\le1\Big\}.
\end{equation}  

 We know from the law of large numbers of Kurtz \cite{Kurtz1978}, see also \cite{BP}, that under mild assumptions on the rates $\beta_j$, $1\leq j\leq k$, for all $T>0$, $Z^{N, z}(t)$ converges to $Y^{z}(t)$ almost surely  and uniformly over the interval $[0,T]$, where $Y^{z}(t)$ is the solution of the ODE
 \begin{equation}\label{ODE}
Y^{z}(t)  :=  Y(t):=z +\int_{0}^{t} b(Y^{z}(s)) ds
\end{equation}
with
\begin{equation}\label{defb}
b(y):=\sum_{j=1}^{k} \beta_{j}(y) h_{j},
\end{equation}
and $Y^z_t$ takes its values in the set
\begin{equation*}
A=\Big\{z\in\mathbb{R}^{d}_{+}, \sum_{i=1}^{d}z_{i}\leqslant1\Big\}.
\end{equation*} 

 We assume that the dynamical system \eqref{ODE} has possibly several equilibria which are locally stable and we consider the basin of attraction $O$ of a stable equilibrium $z^{\ast}$. For the models we have in mind (see the four examples in section \ref{sec:appl} below), $O$ has a characteristic boundary which is either the set $\{z\in A;\ z_1=0\}$, or else the part of the boundary of $O$ which is included in $\mathring{A}$, the interior of $A$. In both cases, that characteristic boundary is denoted $\widetilde{\partial O}$.  The solution of the dynamical system \eqref{ODE} starting from $z\in\widetilde{\partial O}$  remain in $\widetilde{\partial O}$ for all time, since for all $z\in\widetilde{\partial O}$
\begin{equation}\label{condboundcarac}
<b(z),n(z)>  := 0,
\end{equation}
where $n(z)$ is the unit outward normal to $\widetilde{\partial O}$ at $z$.

 Our goal in this paper is to study the most probable location of the exit point in $\widetilde{\partial O}$ of the trajectory of $Z^{N,z}$, for large $N$, when the starting point  $z$  belongs to the interior of $O$.  In other words, the aim is to study the probability that a trajectory of $Z^{N}$ exits $O$ in the neighborhood of a point $y\in\widetilde{\partial O}$, $\mathbb{P}(|Z^{N,z}(\tau^{N}_{O})-y|<\delta)$ for large $N$.  Here $\tau^{N,z}_{O}$ denotes the first time of exit of the process $Z^{N, z}(t)$ from $O$.  

 We adopt the approach of Day \cite{day1990large}. First we define a reflected Poissonian SDE for which the large deviation principle is satisfied, with the same rate function as the original one defined by \eqref{EqPoisson1}. We then follow the arguments of Day \cite{day1990large} to obtain our results.

 Our motivation comes from the epidemics models. More precisely, we want to understand how the extinction of an endemic situation happens in infectious disease models. 

 In section \ref{reflected}, we define our reflecting Poissonian SDE and we formulate a large deviation principle satisfied by the latter. Section \ref{lemmas} presents some preliminary lemmas about the rate function. These lemmas are mostly adapted versions of those in chapter 6 of \cite{freidlin2012random} and in Day \cite{day1990large}. Section \ref{exitposition} discusses the large deviations of the exit measure. Finally, in section \ref{sec:appl}, we consider four distinct epidemics models. In all those four cases, with the help of the Pontryaguin maximum principle, we show that for large $N$, the process $Z^N_t$ hits $\widetilde{\partial O}$ in the vicinity of one particular point of that boundary with a probability close to one. 

\section{Notation and the main assumption}
 We define the following cone generated by the family of vectors $(h_{j})_{j=1,...,k}$ 
\begin{equation}\label{Cspan}
\mathcal{C}=\Big\{y\in\mathbb{R}^{d}: y=\sum_{j=1}^{k}\mu^{j}h_{j}, \mu^{j}\geq0\quad\forall j\Big\}.
\end{equation}
We remark that in all the epidemics models that we will consider, the family of vectors $(h_{j})_{j=1,...,k}$ is such that
$\mathcal{C}=\mathbb{R}^{d}$.

 We now formulate some assumptions which are useful in order to establish the large deviation principle of the reflected Poisson SDE that we will construct in section~\ref{reflected}, and which we will assume o hold throughout this paper, without recalling them in the statements. 
\begin{assumption}\label{assump2}
\begin{enumerate}
\item $\mathcal{C}=\mathbb{R}^{d}$.\label{assump20}
\item $\bar{O}$ is compact and there exists a point $z_{0}$ in the interior of $O$ such that each segment joining $z_{0}$ and any $z\in\partial O$ does not touch any other point of the boundary $\partial O$.  \label{assump21}
\item For each $a>0$ small enough, $z\in\bar{O}$, if we denote $z^{a}=z+a(z_{0}-z)$, there exist two positive constants $c_{1}$ and $c_{2}$ such that  \label{assump22}
 \begin{align*}
   |z-z^{a}| &\leq c_{1} a  \\
   \dist(z^{a},\partial O) &\geq c_{2} a  
\end{align*}
 \item The rate functions $\beta_{j}$ are Lipschitz continuous with the Lipschitz constant equal to $C$.  \label{assump23}
 \item There exist two constants $\lambda_{1}$ and $\lambda_{2}$ such that whenever $z\in \bar{O}$ is such that $\beta_{j}(z)<\lambda_{1}$, $\beta_{j}(z^{a})>\beta_{j}(z)$ for all $a\in]0, \lambda_{2}[$ . \label{assump24}
 \item There exist  $\nu\in]0,1/2[$ and $a_0>0$  such that 
 $C_{a}\geq\exp\{-a^{-\nu}\}$ for all $0<a<a_0$,
 where
 \begin{equation*}
 C_{a}=\inf_{j}\inf_{z:\dist(z,\partial O)\geq a}\beta_{j}(z).
 \end{equation*}
 \label{assump25}
 \end{enumerate}
\end{assumption}

\section{Reflected solution of a Poisson driven SDE and Large deviation Principle($LDP$)}\label{reflected}

 For any $z\in\bar{O}$, let
\begin{equation}\label{initialcond}
z^{N}=
\begin{cases}
      \frac{[N z]}{N}& \text{ if } \frac{[N z]}{N}\in\bar{O}, \\
      \underset{y\in\bar{O}^{(N)}}{\arg\inf}~~|y-z|& \text{otherwise},
\end{cases}
\end{equation}
and $\tilde{Z}^{N}_{t}$ denote the d-dimensional processes  defined by
\begin{align}\label{reflexion}
\tilde{Z}^{N}(t)=\tilde{Z}^{N, z}(t)&:=z^{N}+\frac{1}{N}\sum_{j=1}^{k}h_{j}Q_{t}^{N,j}-\frac{1}{N}\sum_{j=1}^{k}h_{j}\int_{0}^{t}\mathbf{1}_{\{\tilde{Z}^{N,z}(s-)+\frac{h_{j}}{N}\not\in\bar{O}\}}dQ_{s}^{N,j} \nonumber \\
&:=z^{N}+\frac{1}{N}\sum_{j=1}^{k}h_{j}\int_{0}^{t}\mathbf{1}_{\{\tilde{Z}^{N,z}(s-)+\frac{h_{j}}{N}\in\bar{O}\}}dQ_{s}^{N,j} , 
\end{align}
where for $j=1,...,k$,  $Q_{.}^{N,j}$ is given as
\begin{equation}\label{Qaux}
Q_{t}^{N,j}=P_{j}\Big(N\int_{0}^{t} \beta_{j}(\tilde{Z}^{N,z}(s)) ds \Big).
\end{equation}
We then obtain a Poisson driven SDE whose solution takes its values in $O^{(N)}=A^{(N)}\cap\bar{O}$.

 Let $D_{T,\bar{O}}$ denote the set of functions from $[0,T]$ into $\bar{O}$ which are right continuous and have left limits, $\mathcal{A}\mathcal{C}_{T,\bar{O}}\subset D_{T, \bar{O}}$ the subspace of absolutely continuous functions. For any $\phi$, $\psi\in D_{T,\bar{O}}$ and $W$ a subset of $D_{T,\bar{O}}$ let
\begin{equation*}
  \|\phi-\psi\|_{T}=\sup_{t\leq T}|\phi_{t}-\psi_{t}|
\end{equation*}
where $|.|$ denotes Euclidian distance in $\mathbb{R}^{d}$ and
\begin{equation}\label{moduleinf}
\rho_{T}(\phi,W)=\inf_{\varphi\in W} \|\phi-\varphi\|_{T}.
\end{equation}

 For all $\phi\in\mathcal{AC}_{T,\bar{O}}$, let $\mathcal{A}_{d}(\phi)$ denote the (possibly empty) set of vector-valued Borel measurable functions $\mu$ such that for all $j=1,...,k$, $0\le t\le T$,
 $\mu^{j}_{t}\geq0$ and 
\begin{equation*}\label{allowed}
  \frac{d\phi_{t}}{dt}=\sum_{j=1}^k \mu^{j}_{t} h_j, \quad\text{t a.e}.
\end{equation*}
We define the rate function
\begin{equation}\label{ratefunct}
I_{T}(\phi) :=
 \begin{cases}
\underset{\mu\in\mathcal{A}_{d}(\phi)}{\inf} I_{T}(\phi|\mu),& \text{ if } \phi\in\mathcal{AC}_{T,\bar{O}}; \\
\infty ,& \text{ else,}
\end{cases}
\end{equation}
where
\begin{equation*}
  I_{T}(\phi|\mu)=\int_{0}^{T}\sum_{j=1}^{k}f(\mu_{t}^{j},\beta_{j}(\phi_{t}))dt
\end{equation*}
with $f(\nu,\omega)=\nu\log(\nu/\omega)-\nu+\omega$.
 We assume in the definition of $f(\nu,\omega)$ that for all $\nu>0$, $\log(\nu/0)=\infty$ and $0\log(0/0)=0\log(0)=0$. By the definition of $f$, there is not difficult to remark  that
\begin{equation} 
I_{T}(\phi)=0 \quad \text{if and only if $\phi$ solves the $ODE$~\eqref{ODE}}.
\end{equation}
 
 The above rate function can also be defined as
 \begin{equation*}
I_{T}(\phi) :=
 \begin{cases}
\int_{0}^{T}L(\phi_{t},\phi'_{t})dt& \text{ if } \phi\in\mathcal{AC}_{T, \bar{O}} \\
\infty & \text{ else.}
\end{cases}
\end{equation*}
where for all $z\in \bar{O}$, $y\in\mathbb{R}^{d}$ 
\begin{equation}\label{functionnalL}
L(z,y)=\sup_{\theta\in\mathbb{R}^{d}}\ell(z,y,\theta)
\end{equation}
with for all $z\in \bar{O}$, $y\in\mathbb{R}^{d}$ and $\theta\in\mathbb{R}^{d}$
\begin{equation*}
\ell(z,y,\theta)=\big<\theta,y\big>-\sum_{j=1}^{k}\beta_{j}(z)(e^{\big<\theta,h_{j}\big>}-1)
\end{equation*}
 
 The rate function defined above is a good rate function (cf \cite{Kratz2014}), that is for all $s>0$, the set $\{\phi\in D_{T,\bar O}:I_{T}(\phi)\leq s\}$ is compact. We now formulate the result concerning the $LDP$ for our reflected model \eqref{reflexion}. This result is proved in \cite{Pard2017}.\begin{theorem}\label{LDP}
Let $\{\tilde{Z}^{N,z}_t,\ 0\le t\le T\}$ be the solution of \eqref{reflexion}.
\begin{description}
  \item[a)] For $z\in \bar{O}$, $\phi\in D_{T,\bar{O}}$, $\phi_{0}=z$, $\eta>0$ and $\delta>0$ there exists $N_{\eta,\delta}\in\mathbb{N}$ such that for all $N>N_{\eta,\delta}$
  \begin{equation*}
  \mathbb{P}_{z}\Big(\|\tilde{Z}^{N}-\phi\|_{T}<\delta\Big)\geq\exp\{-N(I_{T}(\phi)+\eta)\}.
  \end{equation*}
  \item[b)] For any open subset $G$ of $D_{T,\bar{O}}$, the following holds uniformly over $z\in\bar{O}$
  \begin{equation*}
  \liminf_{N\to\infty}\frac{1}{N}\log\mathbb{P}_{z}(\tilde{Z}^{N}\in G)\geq-\inf_{\phi\in G, \phi_{0}=z} I_{T}(\phi).
  \end{equation*}
  \item[c)] For $z\in\bar{O}$, $\delta>0$ let
  \begin{equation*}
   H_{\delta}(s)=\{\phi\in D_{T,\bar{O}}: ,\phi_{0}=z, \rho_{T}(\phi,\Phi(s))\geq\delta\}\quad\text{where}\quad \Phi(s)=\{\psi\in D_{T,\bar{O}}: I_{T}(\psi)\leq s\}.
  \end{equation*}
  For any $\delta, \eta, s>0$ there exists $N_{0}\in\mathbb{N}$ such that for all $N>N_{0}$
  \begin{equation*}
  \mathbb{P}_{z}(\tilde{Z}^{N}\in H_{\delta}(s))\leq\exp\{-N(s-\eta)\}.
  \end{equation*}
  \item[d)] For any closed subset $F$ of $D_{T,\bar{O}}$, the following holds uniformly over $z\in\bar{O}$
  \begin{equation*}
  \limsup_{N\to\infty}\frac{1}{N}\log\mathbb{P}_{z}(\tilde{Z}^{N}\in F)\leq-\inf_{\phi\in F, \phi_{0}=z} I_{T}(\phi).
  \end{equation*}
\end{description}
\end{theorem}

\section{Notations and important Lemmas}\label{lemmas}

We assume from now on that there exists a (unique) point $z^\ast\in O$ such that for any $z\in O$,
$Y^z_t\to z^\ast$, as $t\to\infty$.

 For $z, y \in \bar{O}$, we define the following functionals.
\begin{align}
V_{\bar{O}}(z,y,T)&:= \inf_{\phi\in D_{T,\bar{O}}, \phi_{0}=z, \phi_{T}=y} I_{T}(\phi) \notag \\
V_{\bar{O}}(z,y)&:= \inf_{T>0} V_{\bar{O}}(z,y,T)  \notag \\
 V_{\widetilde{\partial O}}&:= \inf_{y \in \widetilde{\partial O}} V_{\bar{O}}(z^{*},y). \notag
\end{align}
We will denote by $B_{r}(y)$ the open ball centered at $y$ with radius $r$, and  
$B_{r}(K)=\cup_{y\in K}B_r(y)$.
For large $N$, the function $V_{\bar{O}}(z,y)$ quantifies the energy needed for a trajectory to deviate from a solution to the ODE \eqref{ODE},  and go from $z$ to $y$, without leaving $\bar{O}=O\cup 
\widetilde{\partial O}$ and $V_{\widetilde{\partial O}}$ is the minimal energy required to leave the domain $O$ when starting from $z^{*}$.
We now prove a few Lemmas, which are analogues of some Lemmas of chapter 6 in \cite{freidlin2012random}.

\begin{lemma}\label{assumption1}
There exists a constant $C>0$ and a function $\mathcal{K}\in C(\R_+,\R_+)$ with $\mathcal{K}(0)=0$ such that for all $\rho>0$ small enough, there exists a constant $T(\rho)$ with $T(\rho)\le C\rho$ such that for all $x\in\mathring{A}\cap \bar{O}$ and all $z$, $y\in B_{\rho}(x)\cap\mathring{O}$ there exists an curve $(\phi_{t})=(\phi_{t}(\rho,z,y))$ $0\leq t\leq T(\rho)$ with $\phi_{0}=z$, $\phi_{T(\rho)}=y$ entirely in $B_{\rho}(x)\cap\bar{O}$, such that $I_{T(\rho)}(\phi)\le\mathcal{K} (T(\rho))$.
\end{lemma}

\begin{proof}
We will exploit Assumptions \ref{assump2}.\ref{assump22} and \ref{assump2}.\ref{assump25} and refer to the notations there. 
Note that the distance between $y$ and $z$ is at most $2\rho$. Let $y^a$ and $z^a$ be the points defined in Assumption \ref{assump2}.\ref{assump22}, with $a=2\rho/c_2$. Then both are at distance at least $2\rho$ from the boundary of $O$, while they are at distance less than $2\rho$ one from another. Consequently the segment joining those two points is at distance at least $\sqrt{3}\rho$ from the boundary. We choose as function $\phi$ the piecewise linear function which moves at speed one, first in straight line from $z$ to $z^a$, then from $z^a$ to $y^a$, and finally from $y^a$ to $y$. The time needed to do so is bounded by 
$2\left(1+\frac{c_1}{c_2}\right)\rho$. Thanks to Assumption \ref{assump2}.\ref{assump20}, $I_{T(\rho)}(\phi)<\infty$. 
Refering to the formula \eqref{ratefunct} for the rate function and to Assumption  \ref{assump2}.\ref{assump25}, we see that the  contribution of the straight line between $z^a$ and $y^a$ to $I_{T(\rho)}(\phi)$ is bounded by $C\rho^{1-\nu}$, while the contribution of the two other pieces is bounded by a universal constant times (below $\bar c=2\frac{c_1}{c_2}$)
\[\int_0^{\bar c\rho}\frac{dt}{t^\nu}=\frac{\bar c^{1-\nu}}{1-\nu}\rho^{1-\nu}.\]
The result follows.
\end{proof}

\begin{lemma}\label{lemma1.2}
$\forall\eta>0$, $K\subseteq \mathring{O}\cup\widetilde{\partial O}$ compact, there exist $T_{0}$ such that for any $z$, $y\in K$ 
there exists a function $\phi_{t}$, $0\leq t\leq T$ satisfying $\phi_{0}=z$, $\phi_{T}=y$, $T\leq T_{0}$ such that $I_{T}(\phi)\leq V_{\bar{O}}(z,y)+\eta$.
\end{lemma}

\begin{proof}
As $K$ is compact  there exists a finite number $M$ of points $\{z_{i}, 1\leq i\leq M\}$ in $K$ such that 
\begin{equation*}
  K\subseteq\bigcup_{i=1}^{M}B_{r}(z_{i}).
\end{equation*}
 For all $z,y\in K$, there exist $1\leq i,j\leq M$ with $z\in B_{r}(z_{i})$, $y\in B_{r}(z_{j})$. Since
 $V_{\bar{O}}(u,v)$ is continuous, chosing  $r$ small enough, we deduce that
\begin{equation}\label{elemma1.2}
  V_{\bar{O}}(z_{i},z_{j})\leq V_{\bar{O}}(z,y)+\frac{\eta}{8}.
\end{equation}
Moreover, from the finiteness of $V_{\bar{O}}$ we have that for all $z_{i}, z_{j}$, there exists $T^{i,j}$ and $\widetilde{\phi}_{t}$ with $\widetilde{\phi}_{0}=z_{i}$, $\widetilde{\phi}_{T^{i,j}}=z_{j}$
\begin{equation*}
  I_{T^{i,j}}(\widetilde{\phi})\leq V_{\bar{O}}(z_{i},z_{j})+\frac{\eta}{4}.
\end{equation*}
We can fix $T_{0}=\underset{i,j}{\max}T^{i,j}+2$ and Lemma \ref{assumption1} tells us that it is always possible to connect $z$ and $x_{i}$ resp ($z_{j}$ and $y$) with $\phi^{i}_{t}$, $0\leq t\leq T^{i}<1$ resp ($\phi^{j}_{t}$, $0\leq t\leq T^{j}<1$) such that $\phi^{i}_{0}=z$, $\phi^{i}_{T^{i}}=z_{i}$ resp ($\phi^{j}_{0}=z_{j}$, $\phi^{j}_{T^{j}}=y$) and 
\begin{equation*}
I_{T^{i}}(\phi^{i})\leq\frac{\eta}{4} \text{ resp }\Big(I_{T^{j}}(\phi^{j})\leq\frac{\eta}{4}\Big).
\end{equation*}
Concatenating $\phi^i$, $\widetilde{\phi}$ and $\phi^j$, we obtain a trajectory $\phi$ with all the required properties.
\end{proof}

 Now, we define the equivalence relation $"\mathcal{R}"$ in $\bar{O}$ by
\begin{equation*}
z\mathcal{R} y\quad\text{iff}\quad V_{\bar{O}}(z,y)=V_{\bar{O}}(y,z)=0.
\end{equation*}

\begin{lemma}\label{lemequiva}
Let $z\mathcal{R} y$, $z\neq y$. The trajectory $Y^{z}(t)$ of the dynamical system \eqref{ODE} starting from $z$ lies in the set of points $\{v\in\bar{O}: v\mathcal{R} z\}$.
\end{lemma}

\begin{proof}
As $z\mathcal{R} y$ there exists a sequence of functions $\phi_{t}^{n}$, $0\leq t\leq T_{n}$, $\phi_{0}^{n}=z$, $\phi_{T_{n}}=y$, with values in $\bar{O}$ and such that $I_{T_{n}}(\phi^{n})\rightarrow 0$. The $T_{n}$ are bounded from below by a positive constant. Indeed 
there exists $n_{0}\in\mathbb{N}$ such that for all $n\geq n_{0}$
\begin{equation*}
I_{T_{n}}(\phi^{n}|\mu^{n})\leq1.
\end{equation*}
Now either $T_n\ge \sigma^{-1}$, where $\sigma:=\sup_{1\le j\le k,\ z\in A}\beta_j(z)$,
or else from
 the Lemma 1 in \cite{Sam2017}, for all $1\leq j\leq k$,
\begin{equation}\label{sam1}
\int_{0}^{T_{n}}\mu^{n,j}_{t}dt\leq\frac{2}{-\log(\sigma T_{n})}.
\end{equation}
Moreover $\frac{d\phi^{n}_{t}}{dt}=\overset{k}{\underset{j=1}{\sum}}\mu^{n,j}_{t}h_{j}$, hence
\begin{align*}
    |z-y|&\leq \Big|\int_{0}^{T_{n}}\frac{d\phi^{n}_{t}}{dt}dt\Big|  \\
    &\leq \sqrt{d}\sum_{j=1}^{k} \int_{0}^{T_{n}}\mu^{n,j}_{t}dt ,
\end{align*}
thus there exists $1\leq j\leq k$ such that
\begin{equation}\label{sam2}
\frac{1}{k\sqrt{d}}|z-y|\leq \int_{0}^{T_{n}}\mu^{n,j}_{t}dt.
\end{equation}
Now, combining \eqref{sam1} and \eqref{sam2} we deduce that
\begin{equation*}
T_n\ge\sigma^{-1}\exp\left(-\frac{2k\sqrt{d}}{|z-y|}\right)>0,
\end{equation*}
which shows that  $T_{n}\geq T$ for all $n$ and some $T>0$. Now $I_{T}(\phi_{t}^{n})$ converges to $0$ as $n\to\infty$. Therefore, for a constant $s>0$, there exists $n_{0}\in \mathbb{N}$ such that for all $n\geq n_{0}$, $I_{T}(\phi^{n})<s$. By the compactness of the set $\{\psi: I_{T}(\psi)\leq s\}$, there exists a subsequence $(\phi^{n_{k}})_{n_{k}}$ of these functions which converges, uniformly on $[0, T]$, to a function $\phi_{t}$. As $I_{T}$ is lower semicontinuous, we have
\begin{equation*}
0=\liminf_{k\to\infty}I_{T}(\phi^{n_{k}})\geq I_{T}(\phi).
\end{equation*}
 Thus $I_{T}(\phi)=0$ and $\phi$ is the trajectory of the dynamical system \eqref{ODE} starting from $z$. The points $\phi_{t}$, $0\leq t\leq T$, are equivalent to $z$ and $y$ since we have $V_{\bar{O}}(z,\phi_{t}^{n})$ and $V_{\bar{O}}(\phi_{t}^{n}, y)$ do not exceed $I_{T_{n}}(\phi^{n})\rightarrow 0$ as $n\rightarrow \infty$.

 Let $u$ be one of the points $z$ and $y$ such that that $|\phi_{T}-u|\geq \frac{1}{2}|z-y|$ then $\phi_{T}\mathcal{R} u$. In the same way as earlier, we can find some time interval in which the points of the dynamical system starting from $\phi_{T}$ are equivalent to $u$. We obtain the result by a successive application of the above reasoning. {\color{red} cf. preuve Lemme 1.5 page 165 de FW}
\end{proof}

\begin{lemma}\label{lemesmallint}
Let all points of a compact $K\subseteq \mathring{O}\cup\widetilde{\partial O}$ be equivalent to each other but not equivalent to any other point in $O\cup\widetilde{\partial O}$. For any $\eta>0$, $\delta>0$ and $z,y\in K$ there exists a function $\phi_{t}$, $0\leq t\leq T$, $\phi_{0}=z$, $\phi_{T}=y$, entirely in the intersection of $\bar{O}$ with the $\delta$-neighborhood of $K$ and such that $I_{T}(\phi)<\eta$.
\end{lemma}

\begin{proof}
As $z,y\in K$ there exists a sequence of functions $\phi_{t}^{n}$, $0\leq t\leq T_{n}$, $\phi_{0}^{n}=z$, $\phi_{T_{n}}^{n}=y$, with values in $O\cup\widetilde{\partial O}$ and such that $I_{T_{n}}(\phi^{n})\rightarrow 0$. And then there exists $n_{0}\in\mathbb{N}$ such that for all $n>n_{0}$, $I_{T_{n}}(\phi^{n})<\eta$. If all curves $\phi_{t}^{n}$ with $n>n_{0}$ left the $\delta$-neighborhood of $K$, then they would have a limit point $x$ outside of this $\delta$-neighborhood and we have $V_{\bar{O}}(z,y)=V_{\bar{O}}(y,z)=0$ thus $x$ is equivalent to $z$ and $y$. A contradiction since all points of a compact $K$ are equivalent to each other but not equivalent to any other point in $\bar{O}$.
\end{proof}

\begin{lemma}\label{lemecomp}
Let $K$ be a compact subset of $O\cup\widetilde{\partial O}$ not containing any $\omega$-limit set \footnote{A point $\bar z$ is called an $\omega-$limit point of solution $Y(t,z_0)$ of a dynamical system if there exists a sequence $(t_n)_{n\geq1}$ of time such that $t_n\to\infty$ as $n\to\infty$, for which $$Y(t_n,z_0)\to\infty,~~~n\to\infty$$ holds. The set of all such points of $Y(t,z_0)$ is called $\omega-$limit set of $Y(t,z_0)$ and denote $\omega(z_0)$.} entirely. There exist positive constants $c$ and $T_{0}$ such that for all sufficiently large $N$ and any $T>T_{0}$ and $z\in K$ we have
\begin{equation*}\label{probcomp}
  \mathbb{P}_{z}(\tau^{N}_{K}> T) \leq \exp\{-N.c(T-T_{0})\}
\end{equation*}
where $\tau^{N}_{K}$ is the time of first exit of $\tilde{Z}_{t}^{N}$ from $K$ and under $\mathbb{P}_{z}$, $\tilde{Z}_{t}^{N}$ starts from the point $z^{N}$ defined by \eqref{initialcond}.
\end{lemma}

\begin{proof}
As $K$ does not contain any $\omega$-limit set entirely, we can choose $\delta$ sufficiently small such that the closed $\delta$-neighborhood $K^{\delta}$ does not contain any $\omega$-limit set entirely, either.
For $z\in K^{\delta}$, let 
\begin{equation*}
\tau(z)=\inf\{t>0: Y^{z}(t)\not\in K^{\delta}\}.
\end{equation*}
where $Y^{z}(t)$ is the solution of \eqref{ODE} starting from $z$. We have $\tau(z)<\infty$ for all $z\in K^{\delta}$. By the continuous dependence of a solution on the initial conditions,the function $\tau(z)$ is upper semi-continuous, and then it attains its largest value $T_{1}=\sup_{z\in K^{\delta}}\tau(z)<\infty$.
\par Fix $T_{0}=T_{1}+1$ and let $\mathcal{F}^{K^{\delta}}$ the set of all functions $\phi_{t}$ defined for $0\leq t\leq T_{0}$ and assuming values only in $K^{\delta}$. the set of these functions is closed in the sense of uniform convergence and because $I_{T_0}$ is lower semi-continuous,
\begin{equation*}
I_{0}=\min_{\phi\in\mathcal{F}^{K^{\delta}}}I_{T_{0}}(\phi)\quad\text{is attained on $\mathcal{F}^{K^{\delta}}$}.
\end{equation*}
Moreover for all $\phi\in\mathcal{F}^{K^{\delta}}$, $I_{T_{0}}(\phi)>I_{0}>0$ since there are no trajectories of the dynamical system \eqref{ODE} in $\mathcal{F}^{K^{\delta}}$. If a function $\phi$ spend a time $T$ in $K$ with $T$ longer than $T_{0}$ we have $I_{T}(\phi)\geq I_{0}$; for the functions $\phi$ spending time $T\geq2 T_{0}$ in $K$ we have $I_{T}\geq 2I_{0}$, and so on. We deduce that for all $\phi$ spending time $T$ in $T>T_{0}$ in $K$ we have 
\begin{equation*}
  I_{T}(\phi)\geq I_{0}\Big[\frac{T}{T_{0}}\Big]>I_{0}\Big(\frac{T}{T_{0}}-1\Big)=\frac{I_0}{T_0}(T-T_{0}).
\end{equation*}
For $z\in K$ the functions in the set
\begin{equation*}
  \Phi_{z}(I_{0})=\{\phi: \phi_{0}=z, I_{T_{0}}(\phi)\leq I_{0}\}
\end{equation*}
leave $K^{\delta}$ during the time from 0 to $T_{0}$; the trajectories $\tilde{Z}^{N,z}_{t}$ for which $\tau^{N}_{K}>T_{0}$, are at a distance greater than $\delta$ to this set. We deduce by using Theorem \ref{LDP} that for any $z\in K$
\begin{align*}
  \mathbb{P}_{z}(\tau^{N}_{K}> T_{0}) &\leq \mathbb{P}(\rho_{T_{0}}(\tilde{Z}^{N},\Phi_{z}(I_{0}))\geq\delta) \\
   &\leq\exp\{-N(I_{0}-\eta)\}.
\end{align*}
Now we use the Markov property and we have
\begin{align*}
  \mathbb{P}_{z}(\tau^{N}_{K}> (n+1)T_{0}) &\leq \mathbb{E}_{z}\big(\tau^{N}_{K}> nT_{0}; \mathbb{P}_{\tilde{Z}^{N}(nT_{0})}(\tau^{N}_{K}>T_{0})\big) \\
  &\leq  \mathbb{P}_{z}(\tau^{N}_{K}>nT_{0}) \sup_{y\in K}\mathbb{P}_{y}(\tau^{N}_{K}>T_{0}).
\end{align*}
We obtain by induction that
\begin{align*}
  \mathbb{P}_{z}(\tau^{N}_{K}>T) &\leq  \mathbb{P}_{z}\Big(\tau^{N}_{K}>\Big[\frac{T}{T_{0}}\Big]T_{0}\Big)\\
   &\leq \Big(\sup_{y\in K}\mathbb{P}_{y}(\tau^{N}_{K}>T_{0})\Big)^{\big[\frac{T}{T_{0}}\big]} \\
   &\leq \exp\Big\{-N\Big[\frac{T}{T_{0}}\Big](I_{0}-\eta)\Big\}.
\end{align*}
Hence the result with $c=\frac{I_{0}-\eta}{T_0}$, where $\eta$ is an arbitrary small number.
\end{proof}

The following assumptions comes essentially from \cite{freidlin2012random} (page 150).
\begin{assumption}\label{assump32}
 There exists a finite number of compacts $K_{1},\cdots K_{M}\subseteq \widetilde{\partial O}$ such that
\begin{description}
  \item [(1)] $z,y\in K_{i}\quad\text{implies}\quad z\mathcal{R} y$
  \item [(2)] $z\in K_{i}, y\not\in K_{i}\quad\text{implies}\quad z\!\not\!\!{\mathcal{R}} y$ 
 \item [(3)] every $\omega$-limit set of \eqref{ODE} with $Y^{z}(0)\in\widetilde{\partial O}$ is contained entirely in one of the $K_{i}$.
\end{description}
\end{assumption}
 We moreover define $K_{0}=\{z^{*}\}$. 
We now construct as in \cite{day1990large} and  \cite{freidlin2012random} an embedded Markov chain $\tilde{Z}_{n}$ associated to the process $\tilde{Z}^{N}(t)$ in the following way: let $\rho_{0}$ and $\rho_{1}$ such that $0<\rho_{1}<\rho_{0}<\frac{1}{2}\underset{i\neq j}{\min}~~dist(K_{i},K_{j})$,
\begin{align}
  G^{1}_{i}&:= \overline{B_{\rho_{1}}(K_{i})}; \label{setg} \\ 
  C&:= \bar{O}\setminus\bigcup_{0}^{M}B_{\rho_{0}}(K_{i});  \label{setC}\\ 
  \Gamma_{i}&:= B_{2\rho_{0}}(K_{i})\setminus B_{\rho_{0}}(K_{i}); \label{setgamma} \\ 
  \theta_{0}&:= 0; \label{tau0}\\ 
  \sigma_{n}&:= \inf\{t\geq \theta_{n}: \tilde{Z}^{N}(t)\in C\}; \label{sigman}\\ 
  \theta_{n+1}&:= \inf\Big\{t\geq \sigma_{n}: \tilde{Z}^{N}(t)\in\bigcup_{0}^{M}G^{1}_{i}\Big\}; \label{taun} \\  
  \tilde{Z}_{n}&:= \tilde{Z}^{N}(\theta_{n}). \label{Ztilde}
\end{align}
And
\begin{equation}\label{timeexit}
\tau_{O}^{N}:=\inf \Big\{t>0: \tilde{Z}^{N}(t^-)+\frac{1}{N}\sum_{j=1}^{k}h_{j} \Delta Q_{t}^{j}\not\in\bar{O}\Big\}
\end{equation}
where $Q^{j}_{.}$ is defined by \eqref{Qaux} and $\Delta Q_{t}^{j}=Q_{t}^{j}-Q_{t^{-}}^{j}$.

We also introduce the quantities $\tilde{V}_{\bar{O}}(z,y)$, $\tilde{V}_{\bar{O}}(z,K_{i})$, $\tilde{V}_{\bar{O}}(K_{i},K_{j})$ and $V_{\bar{O},K_{0}^{c}}(z,y)$ defined as in \cite{freidlin2012random} by: $\forall z,y\in\bar{O}$
\begin{align*}
\widetilde{V}_{\bar{O}}(z,y) &:= \inf\Big\{I_{T}(\phi):T>0, \phi_{0}=z, \phi_{T}=y,\phi_{t}\in\bar{O}\setminus\bigcup_{\ell=0}^M K_{\ell}~~\forall t\in(0,T)\Big\} \\
\widetilde{V}_{\bar{O}}(z,K_{i})&:=\inf\Big\{I_{T}(\phi):T>0, \phi_{0}=z, \phi_{T}\in K_{i}, \phi_{t}\in\bar{O}\setminus\bigcup_{\ell\neq i}K_{\ell}~~\forall t\in(0,T)\Big\} \\
 \widetilde{V}_{\bar{O}}(K_{i},K_{j}) &:= \inf\Big\{I_{T}(\phi):T>0, \phi_{0}\in K_{i}, \phi_{T}\in K_{j},\phi_{t}\in\bar{O}\setminus\bigcup_{\ell\neq i,j}K_{\ell}~~\forall t\in(0,T)\Big\}\\
V_{\bar{O},K_{0}^{c}}(z,y) &:= \inf\{I_{T}(\phi):T>0,\phi_{0}=z,  \phi_{T}=y,\phi_{t}\in\bar{O}\setminus K_{0}~~\forall t\in(0,T) \} .
\end{align*}

 We now establish the following equality.
\begin{lemma}\label{remark1}
Under the hypothesis \eqref{condboundcarac} of a characteristic boundary, we have
\begin{equation}\label{Vbord}
V_{\widetilde{\partial O}}:=\min_{i=1}^{M}\widetilde{V}_{\bar{O}}(K_{0},K_{i}).
\end{equation}
\end{lemma}

\begin{proof}
We fix $\epsilon>0$ arbitrary. Let $\phi_{t}$, $0\leq t\leq T$ be such that $\phi_{0}=z^{*}$, $\phi_{T}\in\widetilde{\partial O}$ and $I_{T}(\phi)<V_{\widetilde{\partial O}}+\epsilon$. If $\phi_{t}\in\overset{M}{\underset{i=1}{\bigcup}}K_{i}$ for some $t\in [0,T)$, we can find $0\le t_{0}<t_{1}<T$ such that $\phi_{t_{0}}=z^{*}$, $\phi_{t}\not\in \overset{M}{\underset{i=1}{\bigcup}}K_{i}$ for all $t_{0}< t< t_{1}$ and $\phi_{t_{1}}\in K_{j}$ for some $j$. Thus $\widetilde{V}_{\bar{O}}(K_{0},K_{j})\leq V_{\widetilde{\partial O}}+\epsilon$. Otherwise if $\phi_{t}$ has reached $\widetilde{\partial O}$ but avoided $\overset{M}{\underset{i=1}{\bigcup}}K_{i}$, then we can extend $\phi_{t}$ to $t>T$ as a solution of \eqref{ODE}. By the assumption \ref{assump32} (3), $\phi_{t}$ comes arbitrarily close to $\bigcup_{1}^{M}K_{i}$ as $t\rightarrow +\infty$, but without any increase in the value of $I_{T}(\phi)$. It follows that
\begin{equation*}
\min_{i=1}^{M}\widetilde{V}_{\bar{O}}(K_{0},K_{i})\leq V_{\widetilde{\partial O}}+\epsilon.
\end{equation*}
As $\epsilon$ is arbitrary and the reverse inequality, $V_{\widetilde{\partial O}}\leq \overset{M}{\underset{i=1}{\min}}\widetilde{V}_{\bar{O}}(K_{0},K_{i})$, is obvious, we have the result.
\end{proof}

\begin{lemma}\label{esti1}
For all $\eta>0$ there exists $\rho_{0}$ small enough such that for any $\rho_{2}$, $0<\rho_{2}<\rho_{0}$, there exits $\rho_{1}$, $0<\rho_{1}<\rho_{2}$ such that for all $N$ large enough, for all $z$ in the $\rho_{2}$-neighborhood $G^{2}_{i}$ of the compact $K_{i}(i=0,\cdots,M)$,  and all $j\geq0$ we have the inequalities:
\begin{align}\label{inequal1}
  \exp\{-N(\widetilde{V}_{\bar{O}}(K_{i},K_{j})+\eta)\} \leq \mathbb{P}_{z}\{ \widetilde{Z}_{1}\in G^{1}_{j }\}  \leq \exp\{-N(\widetilde{V}_{\bar{O}}(K_{i},K_{j})-\eta)\} 
\end{align}
\end{lemma}

\begin{proof} 
\par By using Lemma \ref{assumption1} let $\rho>0$ such that $T(\rho)<\eta/3\mathcal{K}$. We choose $\rho_{0}>0$ smaller than $\rho/3$ and $\frac{1}{3}\min_{i,j}\dist(K_{i},K_{j})$, and $\rho_{2}\in(0, \rho_{0})$ .
For any two compacts $K_{i}$, $K_{j}$, $i\not=j$, for which $\widetilde{V}_{\bar{O}}(K_{i},K_{j})<\infty$, we choose a function $\phi^{i,j}_{t}$, $0\leq t\leq T_{i,j}$, such that $\phi^{i,j}_{0}\in K_{i}$, $\phi^{i,j}_{T_{i,j}}\in K_{j}$, $\phi^{i,j}_{t}$ does not touch $\underset{\ell\neq i,j}{\bigcup}K_{\ell}$ and for which 
\begin{equation}\label{inequal2}
  I_{T_{i,j}}(\phi^{i,j})\leq\widetilde{V}_{\bar{O}}(K_{i},K_{j})+\frac{\eta}{6}.
\end{equation}
Let 
\begin{equation*}
  \rho_3=\frac{1}{2}\min\left\{\dist\left(\phi^{i,j}_{t},\bigcup_{\ell\neq i,j}K_{\ell}\right): 0\leq t\leq T_{i,j}, i,j=0,...,M\right\}
\end{equation*}
We now choose $0<\rho_1<\min(\rho/2, \rho_2, \rho_3)$, and $\delta$ smaller than $\rho_{1}$, $\rho_{0}-\rho_{2}$. By Lemma \ref{assumption1}, for any $z\in G^{2}_{i}$, let 
$(\psi^{i,1}_{t}$, $0\leq t\leq \bar{T}_{i,1})$ with $\psi^{i,1}_{0}=z$, $\psi^{i,1}_{\bar{T}_{i,1}}=z'\in K_{i}$ such that $\psi^{i,1}$ stays in $G^{2}_{i}$ and
\begin{equation*}
I_{\bar{T}_{i,1}}(\psi^{i,1})\leq\frac{\eta}{6}.
\end{equation*}
We also have $\dist(\psi^{i,1},C)\geq\delta$ where $C$ is defined by \eqref{setC}. Moreover according to Lemma \ref{lemesmallint}, there exists a curve $(\psi^{i,2}_{t})_{t}$, $0\leq t\leq \bar{T}_{i,2}$ in $G^{2}_{i}$ with $\psi^{i,2}_{0}=z'$, $\psi^{i,2}_{\bar{T}_{i,2}}=\phi^{i,j}_{0}\in K_{i}$ such that 
\begin{equation*}
I_{\bar{T}_{i,2}}(\psi^{i,2})\leq\frac{\eta}{6}.
\end{equation*}
We combine these curves with the curve $\phi^{i,j}_{t}$ and we obtain a function $\phi_{t}$, $0\leq t\leq T=\bar{T}_{i,1}+\bar{T}_{i,2}+T_{i,j}$ ($\phi_{t}$ and $T$ depend on $z\in G^{2}_{i}$ and $j$) and $\phi_{0}=z$, $\phi_{T}\in K_{j}$ such that:
\begin{equation*}
  I_{T}(\phi)\leq \widetilde{V}_{\bar{O}}(K_{i},K_{j})+\frac{\eta}{2}.
\end{equation*}
If $j=i$ we define $\phi_{t}$, $0\leq t\leq T$ such that $\phi_{0}=z\in G^{2}_{i}$, $\phi_{T}=z''\in K_{i}$ and  $\dist(z',z'')=\dist(z',K_{i})$ and we have 
\begin{equation*}
  I_{T}(\phi)\leq \frac{\eta}{2}=\widetilde{V}_{\bar{O}}(K_{i},K_{i})+\frac{\eta}{2}.
\end{equation*}
It is easy using Lemma \ref{assumption1} and Lemma \ref{lemma1.2} to justify that the lengths of the intervals of definition of the functions $\phi_{t}$ constructed for all possible compacts $K_{i}$, $K_{j}$ and point $z\in G^{2}_{i}$ can be bounded from above by a constant $T_{0}<\infty$. The functions $\phi_{t}$ can be extended to the intervals from $T$ to $T_{0}$ to be a solution of \eqref{ODE} so that $I_{T}(\phi)=I_{T_{0}}(\phi)$.
\par Any trajectory $\tilde{Z}^{N,z}(t)$ such that $\|\tilde{Z}^{N,z}-\phi\|_{T_{0}}\leq\delta$  reaches the $\delta$-neighborhood of $K_{j}$ without getting closer than $2\rho_{1}-\delta$  to any of the other compacts and then $\tilde{Z}_{1}=\tilde{Z}^{N,z}(\theta_{1})\in G^{1}_{j}$. Thus using  the large deviation Theorem \ref{LDP} , we deduce that there exists $N_{0}$ depending only on $\eta$, $T_{0}$ and $\delta$ such that for all $N\geq N_{0}$ we have 
\begin{align*}
   \mathbb{P}_{z}\{\tilde{Z}_{1}\in G^{1}_{j}\} &\geq \mathbb{P}_{z}(\rho_{T}(\tilde{Z}^{N},\phi)<\delta) \\
   &\geq \exp\{-N(I_{T_{0}}(\phi)+\frac{\eta}{2})\} \\
   &> \exp\{-N(\widetilde{V}_{\bar{O}}(K_{i},K_{j})+\eta)\}.
\end{align*}
And the left inequality of the Lemma follows.

 Using the strong Markov property, it is sufficient to prove the right inequality for $z\in\Gamma_{i}$. With our choice of $\rho_{0}$ and $\delta$, for any curve $\phi_{t}$, $0\leq t\leq T$ beginning in a point of $\Gamma_{i}$, touching the $\delta$-neighborhood of $G^1_{j}$ and not touching the compacts $K_{\ell}$, $\ell\neq i,j$ we have
\begin{equation*}
  I_{T}(\phi)\geq \widetilde{V}_{\bar{O}}(K_{i},K_{j})-\eta/2.
\end{equation*} 
Using lemma \ref{lemecomp}, there exists two constants $c$ and $T_{1}$ such that for all  $N$ large enough and $z\in (O\cup\widetilde{\partial O})\setminus G$ where $G=\bigcup_{i=1}^{M}G^{1}_{i}$, we have:
\begin{equation*}
  \mathbb{P}_{z}(\theta_{1}>T)\leq\exp\{-N c(T-T_{1})\}\quad\text{for all $T>T_{1}$}.
\end{equation*}
Now we fix a $T>T_{1}$ then any trajectory of $\tilde{Z}^{N,z}(t)$ beginning at a point $z\in\Gamma_{i}$ and being in $G^{1}_{j}$ at time $\theta_{1}$ and not touching the compacts $K_{\ell}$, $\ell\neq i,j$  either spends time $T$ without touching $G\cup\widetilde{\partial O}$ (i.e the event $\{\theta_{1}>T\}$ is realized) or reaches $G^{1}_{j}$ before time $T$ and in this second case, with the notation $\Phi_z(s)=\{\psi\in D_{T,\bar{O}}: \psi(0)=z,\ I_{T}(\psi)\leq s\}$, the event \begin{equation*}
 \big\{\rho_{T}(\tilde{Z}^{N,z},\Phi_{z}(\widetilde{V}_{\bar{O}}(K_{i},K_{j})-\eta/2))\geq\delta\big\} \text{ is realized}.
 \end{equation*}
Hence, for any $z\in\Gamma_{i}$ we have from Lemma \ref{lemecomp} and Theorem \ref{LDP} (c) that for
$N$ large enough
\begin{align*}
  \mathbb{P}_{z}(\tilde{Z}^{N}(\theta_{1})\in G^{1}_{j}) &\leq\mathbb{P}_{z}(\theta_{1}>T) 
   + \mathbb{P}_{z}(\rho_{T}(\tilde{Z}^{N},\Phi_{z}(\widetilde{V}_{\bar{O}}(K_{i},K_{j})-\eta/2))\geq\delta)\\
   &\leq \exp\{-N c(T-T_{1})\}+\exp\{-N(\widetilde{V}_{\bar{O}}(K_{i},K_{j})-2\eta/3)\}\\
   &\leq 2\exp\{-N(\widetilde{V}_{\bar{O}}(K_{i},K_{j})-2\eta/3)\}\quad\text{with $T$ large enough} \\
   &\leq \exp\{-N(\widetilde{V}_{\bar{O}}(K_{i},K_{j})-\eta)\}\quad\text{$N$ large enough such that $\frac{\ln(2)}{N}<\frac{\eta}{3}$} 
\end{align*}
The lemma is proved.
\end{proof}

 Before we present others lemmas which will be useful, let us define the sequence $\kappa_{n}\in\mathbb{N}$ for which $\tilde{Z}_{\kappa_{n}}=\tilde{Z}^{N,z}(\theta_{\kappa_{n}})\in G^{1}_{0}$. Note that the $\kappa_{n}$ are $\bar{Z}_{n}$ stopping times and that the $\theta_{n}$ and $\theta_{\kappa_{n}}$ are $\tilde{Z}^{N,z}(t)$ stopping times.

\begin{lemma} \label{Limit}
For all $z\in O$,
\begin{equation*}
\lim_{N \rightarrow \infty} \mathbb{P}_{z}\big(\tilde{Z}^{N}(\theta_{\kappa_{1}}\wedge\tau_{O}^{N}) \in G^{1}_{0}\big) =1.
\end{equation*}
\end{lemma}

\begin{proof}
We prove that 
\begin{equation*}
\lim_{N \rightarrow \infty} \mathbb{P}_{z}\big(\tilde{Z}^{N}(\theta_{\kappa_{1}}\wedge\tau_{O}^{N})\not \in G^{1}_{0}\big) =0.
\end{equation*}
It is enough to take $z\in O \setminus G^{1}_{0}$.  Indeed if $z\in G^{1}_{0}$ by the strong Markov property,
\begin{align*}
  \mathbb{P}_{z}\big(\tilde{Z}^{N}(\theta_{\kappa_{1}}\wedge\tau_{O}^{N})\not \in G^{1}_{0}\big)  &=\mathbb{E}_{z}\Big(\mathbb{E}_{z}\big(\mathbf{1}_{\{\tilde{Z}^{N}(\theta_{\kappa_{1}}\wedge\tau_{O}^{N})\not \in G^{1}_{0}\}}|\mathcal{F}_{\sigma_{0}}\big)\Big)   \\
    &=\mathbb{E}_{z}\Big(\mathbb{P}_{\tilde{Z}^{N}(\sigma_{0})}\big(\tilde{Z}^{N}(\theta_{\kappa_{1}}\wedge\tau_{O}^{N})\not \in G^{1}_{0}\big)\Big) \\
    &\leq\sup_{v\in \Gamma_{0}}\mathbb{P}_{v}\big(\tilde{Z}^{N}(\theta_{\kappa_{1}}\wedge\tau_{O}^{N})\not \in G^{1}_{0}\big)
  \end{align*}
  
 Let furthermore $T:=\inf\{ t\geq 0 | Y^{z}(t) \in B_{\rho_{1}/2}(z^\ast)\}$. Since $Y^{z}$ is continuous and never reaches $\widetilde{\partial O}$, we have $\inf_{t\geq 0} \dist(Y^{z}(t),\widetilde{\partial O}) =:\delta >0$. Hence we have the following implication:
\begin{equation*}
\sup_{t \in [0,T]} | \tilde{Z}^{N,z}_t - Y^{z}(t)|\leq \frac{\delta}{2} \Rightarrow \tilde{Z}^{N,z}(\theta_{\kappa_{1}}\wedge\tau_{O}^{N})\in G^{1}_{0}.
\end{equation*}
In other words,
\begin{equation}
\mathbb{P}_{z}\big(\tilde{Z}^{N}(\theta_{\kappa_{1}}\wedge\tau_{O}^{N}) \notin G^{1}_{0}\big)
\leq \mathbb{P}_{z}\Big(\sup_{t \in [0,T]} | \tilde{Z}^{N,z}(t) -Y^{z}(t)| > \frac{\delta}{2}\Big). \label{ConvergenceA.s.}
\end{equation}
The right hand side of Inequality \eqref{ConvergenceA.s.} converges to zero as $N \rightarrow \infty$ by the weak law of large numbers established in \cite{Pard2017}.
\end{proof}

 In the following, we present some lemmas which are analogue to the lemmas of \cite{day1990large}.
\begin{lemma}\label{lemma2.2}
Given $\eta>0$, there exists $\rho_{0}>0$ (which can be chosen arbitrarily small) such that for any $\rho_{2}\in(0,\rho_0)$,  there exists $\rho_{1}\in(0, \rho_{2})$  and N large enough such that for all $z\in G^{1}_{0}$
\begin{equation*}
\exp(-N(V_{\widetilde{\partial O}}+\eta))\leq \mathbb{P}_{z}\left(\tilde{Z}_{1}\in \bigcup_{i=1}^{M}G^{1}_{i}\right)\leq \exp(-N(V_{\widetilde{\partial O}}-\eta)).
\end{equation*}
\end{lemma}

\begin{proof}
We have
\begin{equation*}
  \mathbb{P}_{z}\left(\tilde{Z}_{1}\in \bigcup_{i=1}^{M}G^{1}_{i}\right)=\sum_{i=1}^{M}\mathbb{P}_{z}\left(\tilde{Z}_{1}\in G^{1}_{i}\right)
\end{equation*}
then we deduce from Lemma \ref{esti1} that $0<\rho_1<\rho_2<\rho_0$ can be chosen in such a way that $\forall x\in G^{1}_{0}$,
\begin{align*}
  &\exp\{-N\eta/2\}\sum_{i=1}^{M}\exp\{-NV_{\bar{O}}(K_{0},K_{i})\}\leq \mathbb{P}_{z}\left(\tilde{Z}_{1}\in \bigcup_{i=1}^{M}G^{1}_{i}\right), \\
  &\mathbb{P}_{z}\left(\tilde{Z}_{1}\in \bigcup_{i=1}^{M}G^{1}_{i}\right)\leq \exp\{N\eta/2\}\sum_{i=1}^{M}\exp\{-NV_{\bar{O}}(K_{0},K_{i})\}. 
 \end{align*} 
 Moreover it is easy to see that 
 \[ \exp\{-N\min_{i}V_{\bar{O}}(K_{0},K_{i})\}\le\sum_{i=1}^{M}\exp\{-NV_{\bar{O}}(K_{0},K_{i})\}\le M\exp\{-N\min_{i}V_{\bar{O}}(K_{0},K_{i})\}.\]
Hence $\forall x\in G^{1}_{0}$,
\begin{align*}
 & \exp\Big\{-N\Big(\min_{i}V_{\bar{O}}(K_{0},K_{i})+\eta/2\Big)\Big\}\leq \mathbb{P}_{z}\left(\tilde{Z}_{1}\in \bigcup_{i=1}^{M}G^{1}_{i}\right), \\
 &\mathbb{P}_{z}\left(\tilde{Z}_{1}\in \bigcup_{i=1}^{M}G^{1}_{i}\right)\leq \exp\Big\{-N\Big(\min_{i}V_{\bar{O}}(K_{0},K_{i})-\eta/2-\frac{1}{N}\log M\Big)\Big\}.
 \end{align*}
 The result now follows from Lemma \ref{remark1} if we choose $N$ such that 
\begin{equation*}
 N>\frac{2|\log(M)|}{\eta}.
\end{equation*}

\end{proof}

In the following Lemma, we establish that an exit to the characteristic boundary  $\widetilde{\partial O}$ is relatively likely once $\tilde{Z}^{N,z}(t)$ is close to it.

\begin{lemma}\label{lemma2.3}
Given $\eta>0$, there exist $0<\gamma<\eta$, $0<\rho<\eta$, and $N_{0}$ large enough so that whenever $\dist(z,\widetilde{\partial O})<\rho/3$ and $N>N_{0}$,
\begin{equation*}
\mathbb{P}_{z}\left(\tau_{O}^{N}<\gamma; \sup_{0\leq t\leq \gamma}|\tilde{Z}^{N}(t)-z|<\rho\right)\geq \exp(-N\eta).
\end{equation*}
\end{lemma}

\begin{proof}
Let 
\begin{equation*}
G=\big\{\phi\in D_{\gamma,A}: \dist(\phi_{0},\widetilde{\partial O})<\rho/3, \sup_{0\leq t\leq \gamma}|\phi_{t}-\phi_{0}|<\rho/2\quad\text{and}\quad\inf\{t: \phi_{t}\in\bar{O}^{c}\}<\gamma\big\}.
\end{equation*}
 $G$ is  open for the Skorohod topology and as $Z^{N,z}$ satisfies the large deviation principle, we deduce that for all $z\in O$ such that $\dist(z,\widetilde{\partial O})<\rho/3$ we have for $N$ large enough
\begin{equation*}
\mathbb{P}_{z}(Z^{N}\in G)\geq \exp\{-N(\inf_{\phi\in G}I_{\gamma}(\phi)+\eta/4)\}.
\end{equation*}
Moreover from Lemma \ref{assumption1} we can choose $\gamma, \rho<\eta$ such that $\inf_{\phi\in G}I_{\gamma}(\phi)<\eta/4$.
Consequently
\begin{equation*}
\mathbb{P}_{z}(Z^{N}\in G)\geq \exp\{-N\eta/2\}.
\end{equation*}
We also have for all $z\in O$ with $\dist(z,\widetilde{\partial O})<\rho/3$.
\begin{align}\label{aligne1}
    &\mathbb{P}_{z}\Big(\tau^{N}_{O}<\gamma, \sup_{0\leq t\leq \gamma}|\tilde{Z}^{N}(t)-z|<\rho\Big) \nonumber \\
    &\geq\mathbb{P}_{z}\Big(\tau^{N}_{O}<\gamma, \sup_{\tau^{N}_{O}\leq t\leq \gamma}|\tilde{Z}^{N}(t)-z|<\rho, \sup_{0\leq t<\tau^{N}_{O}}|\tilde{Z}^{N}(t)-z|<\rho/2\Big) \nonumber\\
    &=\mathbb{P}_{z}\Big(\sup_{\tau^{N}_{O}\leq t\leq \gamma} |\tilde{Z}^{N}(t)-z|<\rho \Big|\tau^{N}_{O}<\gamma, \sup_{0\leq t<\tau^{N}_{O}}|\tilde{Z}^{N}(t)-z|<\rho/2\Big) \nonumber\\
    &\times\mathbb{P}_{z}\Big(\tau^{N}_{O}<\gamma, \sup_{0\leq t<\tau^{N}_{O}}|\tilde{Z}^{N}(t)-z|<\rho/2\Big).   
\end{align}
Moreover we have
\begin{align}\label{aligne2}
    \mathbb{P}_{z}\Big(\tau^{N}_{O}<\gamma, \sup_{0\leq t< \tau^{N}_{O}}|\tilde{Z}^{N}(t)-z|<\rho/2\Big)&=\mathbb{P}_{z}\Big(\tau^{N}_{O}<\gamma, \sup_{0\leq t< \tau^{N}_{O}}|Z^{N}(t)-z|<\rho/2\Big) \nonumber  \\
    &\geq \mathbb{P}_{z}(Z^{N}\in G)\geq \exp\{-N\eta/2\}.
\end{align}
Using the strong Markov property, we also have
\begin{align}\label{aligne3}
  & \mathbb{P}_{z}\Big(\sup_{\tau^{N}_{O}\leq t\leq \gamma}|\tilde{Z}^{N}(t)-z|<\rho \Big|\tau^{N}_{O}<\gamma, \sup_{0\leq t<\tau^{N}_{O}}|\tilde{Z}^{N}(t)-z|<\rho/2\Big)  \nonumber  \\
  &= \mathbb{E}_z\Big(\mathbb{P}_{\tilde{Z}^N(\tau^N_O)}\Big(\sup_{\tau^N_O\leq t\leq \gamma}|\tilde{Z}^{N}(t)-z|<\rho\Big)\Big|\tau^{N}_{O}<\gamma, \sup_{0\leq t<\tau^{N}_{O}}|\tilde{Z}^{N}(t)-z|<\rho/2\Big)  \nonumber  \\
    & \geq\inf_{y:|y-z|<\rho/2}\mathbb{P}_{y}\Big\{\sup_{0\leq t\leq \gamma}|\tilde{Z}^{N}(t)-z|<\rho\Big) \nonumber \\
    & \geq\inf_{y:|y-z|<\rho/2}\mathbb{P}_{y}\Big(\sup_{0\leq t\leq \gamma}|\tilde{Z}^{N}(t)-y|<\rho/2\Big) \nonumber \\
    & \geq \exp\{-N\eta/2\}, 
\end{align}
where the last inequality follows from Theorem \ref{LDP} and Lemma \ref{assumption1}.
Combining \eqref{aligne1}, \eqref{aligne2} and \eqref{aligne3} we have the result.
\end{proof}

 In the next lemmas we denote
\begin{equation*}
O_{\epsilon}:=\{z\in O: \dist(z, \widetilde{\partial O})>\epsilon\}
\end{equation*}
 and we establish some inequalities involving $\theta_{\kappa_{1}}$. 

\begin{lemma}\label{lemma2.4}
For all  $\eta, \rho_{0}>0$, $y\in \widetilde{\partial O}$, and $z\in O\setminus \bigcup_{0}^{M}B_{\rho_{0}}(K_{i})$, there exist $\delta_{0}$, $\rho_{2}$ and $N_{0}$ so that for all $\rho_{1}<\rho_{2}$, for all $\delta<\delta_{0}$ and $N>N_{0}$
\begin{equation*}
\mathbb{P}_{z}\left(\tau^{N}_{O}<\theta_{\kappa_{1}}; |\tilde{Z}^{N}(\tau^{N}_{O})-y|<\delta\right)\geq \exp\big(-N(V_{\bar{O},K_{0}^{c}}(z,y)+\eta)\big)
\end{equation*}
\end{lemma} 

\begin{proof}
First, by definition of $V_{\bar{O},K_{0}^{c}}(z,y)$, there exist $\bar{T}$ and a curve $\psi_{t}$, $0\leq t\leq \bar{T}$  with $\psi_{0}=z$, $\psi_{\bar{T}}=y$ and $\psi_{t}\in \bar{O}\setminus K_{0}$ for all $0\leq t\leq \bar{T}$ such that
\begin{equation*}
  I_{\bar{T}}(\psi)\leq V_{K_{0}^{c}}(z,y)+\frac{\eta}{12}.
\end{equation*}
We can next choose $a$ such that $\psi^{a}_{t}$, $0\leq t\leq \bar{T}$ defined by 
\begin{equation*}
\psi^{a}_{t}=(1-a)\psi_{t}+a z_{0}
\end{equation*} 
satisfies $\dist(\psi^{a}_{t},\widetilde{\partial{O}})\geq c_{2}a$ and
\begin{equation*}
  I_{\bar{T}}(\psi^{a})\leq I_{\bar{T}}(\psi)+\frac{\eta}{12}.
\end{equation*}
According to the Lemma \ref{assumption1}, there exist two functions $\psi^{a,1}_{t}$, $0\leq t\leq T^{a,1}$ and $\psi^{a,2}_{t}$, $0\leq t\leq T^{a,2}$ such that
$\psi^{a,1}_{0}=z$, $\psi^{a,1}_{T^{a,1}}=\psi^{a}_{0}$, $\psi^{a,2}_{0}=\psi^{a}_{\bar{T}}$, $\psi^{a,2}_{T^{a,2}}=y$ and
\begin{equation*}
I_{T^{a,1}}(\psi^{a,1})\leq\frac{\eta}{12}\quad\text{and}\quad I_{T^{a,2}}(\psi^{a,2})\leq\frac{\eta}{12}.
\end{equation*}
Combining $\psi^{a,1}_{t}$, $0\leq t\leq T^{a,1}$, $\psi^{a}_{t-T^{a,1}}$, $T^{a,1}\leq t\leq \bar{T}+T^{a,1}$ and $\psi^{a,2}_{t-\bar{T}-T^{a,1}}$, $\bar{T}+T^{a,1}\leq t\leq T^{a,2}+\bar{T}+T^{a,1}$ we obtain a function $\varphi_{t}$, $0\leq t\leq T=\bar{T}+T^{a,1}+T^{a,2}$
\begin{equation}\label{ineq3311}
 I_{T}(\varphi)\leq V_{K_{0}^{c}}(z,y)+\frac{\eta}{3}.
\end{equation}

 Let $\rho_{1}<\rho_{0}$, $\delta<\min\Big\{\frac{1}{2}c_{2} a,\frac{ \rho_0}{2}\Big\}$ and $H\subseteq D_{T,\bar{O}}$ be the set of functions $\phi$ having the following properties:
\begin{itemize}
  \item $|\phi_{T}-y|<\frac{1}{2}\delta$
  \item $\phi_{t}$ does not intersect $\big(\bar{O}\setminus O_{\frac{1}{2}\delta}\big) \setminus B_{\frac{1}{2}\delta}(y)$
  \item $\phi_{t}$ does not intersect $\overline{B_{\rho_{1}}(z^{*})}$
\end{itemize}
 $H$ is open and $\varphi\in H$ (if $\varphi$ intersects $\overline{B_{\rho_{1}}(z^{*})}$ we use again Lemma \ref{assumption1} to modify $\varphi$). From theorem \ref{LDP}, for all $N$ large enough we deduce by using \eqref{ineq3311} that
\begin{align*}
  \log\mathbb{P}_{z}(\tilde{Z}^{N}\in H)&\geq -N\Big(\inf_{\psi\in H, \psi_{0}=z}I_{T}(\psi)+\frac{\eta}{3}\Big)\\
  &\geq -N\Big(V_{K_{0}^{c}}(z,y)+\frac{2\eta}{3}\Big).
\end{align*}
Moreover we remark that $\tilde{Z}^{N,z}\in H$ implies $|\tilde{Z}^{N,z}(\tau^{N}_{O}\wedge T)-y|<\frac{1}{2}\delta$ and $T<\theta_{\kappa_{1}}$. So
\begin{equation}\label{inequal3}
\log\mathbb{P}_{z}(|\tilde{Z}^{N}(\tau^{N}_{O}\wedge T)-y|<\frac{1}{2}\delta, T<\theta_{\kappa_{1}})\geq -N\Big(V_{K_{0}^{c}}(z,y)+\frac{2\eta}{3}\Big).
\end{equation}
We also have
\begin{align*}
\Lambda&=\mathbb{P}_{z}\Big(|\tilde{Z}^{N}(\tau^{N}_{O})-y|<\delta, \theta_{\kappa_{1}}>\tau_{O}^{N}\Big)\\
&\geq \mathbb{P}_{z}\Big(|\tilde{Z}^{N}(\tau^{N}_{O}\wedge T)-y|<\frac{1}{2}\delta, |\tilde{Z}^{N}(\tau^{N}_{O})-y|<\delta, T<\tau^{N}_{O}<\theta_{\kappa_{1}}\Big) \\
&+\mathbb{P}_{z}\Big(|\tilde{Z}^{N}(\tau^{N}_{O}\wedge T)-y|<\frac{1}{2}\delta, |\tilde{Z}^{N}(\tau^{N}_{O})-y|<\delta,  \tau^{N}_{O}\leq T<\theta_{\kappa_{1}}\Big) \\
&+\mathbb{P}_{z}\Big(|\tilde{Z}^{N}(\tau^{N}_{O}\wedge T)-y|<\frac{1}{2}\delta, |\tilde{Z}^{N}(\tau^{N}_{O})-y|<\delta, \tau^{N}_{O}<\theta_{\kappa_{1}}\leq T\Big) \\
&\geq\Lambda_{1}+\Lambda_{2}+\Lambda_{3}.
\end{align*}
Further
\begin{align*}
  \Lambda_{1}
  &=\mathbb{E}_{z}\Big(\mathbb{P}_{z}\big(|\tilde{Z}^{N}( T)-y|<\frac{1}{2}\delta, \tau^{N}_{O}<\theta_{\kappa_{1}}, T<\theta_{\kappa_{1}}\wedge\tau^{N}_{O}, |\tilde{Z}^{N}(\tau^{N}_{O})-y|<\delta|\mathcal{F}_{T}\big)\Big)\\
    &=\mathbb{E}_{z}\Big(|\tilde{Z}^{N}( T)-y|<\frac{1}{2}\delta, T<\theta_{\kappa_{1}}\wedge\tau^{N}_{O}, \mathbb{P}_{\tilde{Z}^{N}(T)}\big(\tau^{N}_{O}<\theta_{\kappa_{1}},  |\tilde{Z}^{N}(\tau^{N}_{O})-y|<\delta\big)\Big).
  \end{align*}
Furthermore for any $x$ such that $\dist(x,y)<\delta/2$
\begin{equation*}
  \mathbb{P}_{x}\Big(|\tilde{Z}^{N}(\tau^{N}_{O})-y|<\delta, \tau_{O}^{N}<\theta_{\kappa_{1}}\Big)\geq\mathbb{P}_{x}\Big(\tau_{O}^{N}<\gamma, \sup_{0\leq t\leq \gamma}|\tilde{Z}^{N}(t)-y|\leq\delta\Big).
\end{equation*}
for all $\gamma>0$.  In particular with $\gamma$ selected as in Lemma \ref{lemma2.3},
we deduce that
\begin{equation*}
  \Lambda_{1}\geq\mathbb{P}_{z}\Big(|\tilde{Z}^{N}(T)-y|<\frac{1}{2}\delta, T<\theta_{\kappa_{1}}\wedge\tau^{N}_{O}\Big)\exp(-N\eta/3).
\end{equation*}
We also have
\begin{align*}
  \Lambda_{2}&=\mathbb{P}_{z}\Big(|\tilde{Z}^{N}( \tau^{N}_{O})-y|<\frac{1}{2}\delta, \tau^{N}_{O}\leq T<\theta_{\kappa_{1}}\Big) \\
  &\geq \mathbb{P}_{z}\Big(|\tilde{Z}^{N}(\tau^{N}_{O})-y|<\frac{1}{2}\delta, \tau^{N}_{O}\leq T<\theta_{\kappa_{1}}\Big)\exp(-N\eta/3),
\end{align*}
since
\begin{equation*}
  \Big\{|\tilde{Z}^{N,z}(\tau^{N}_{O})-y|<\frac{1}{2}\delta, \tau^{N}_{O}\leq T<\theta_{\kappa_{1}}\Big\}\subseteq\Big\{|\tilde{Z}^{N,z}(\tau^{N}_{O})-y|<\delta, \tau_{O}^{N}<\theta_{\kappa_{1}}\Big\},
\end{equation*}
and $\Lambda_{3}>0$. Thus
\begin{align}\label{inequal4}
  \Lambda&\geq\mathbb{P}_{z}\Big(|\tilde{Z}^{N}(\tau^{N}_{O}\wedge T)-y|<\frac{1}{2}\delta, T<\theta_{\kappa_{1}}\wedge\tau^{N}_{O}\Big) \exp(-N\eta/3) \nonumber\\
  &+\mathbb{P}_{z}\Big(|\tilde{Z}^{N}(\tau^{N}_{O}\wedge T)-y|<\frac{1}{2}\delta, \tau^{N}_{O}\leq T<\theta_{\kappa_{1}}\Big) \exp(-N\eta/3) \nonumber\\
  &\geq\mathbb{P}_{z}\Big(|\tilde{Z}^{N}(\tau^{N}_{O}\wedge T)-y|<\frac{1}{2}\delta, T<\theta_{\kappa_{1}}\Big)\exp(-N\eta/3).
\end{align}
It follows from \eqref{inequal3} and \eqref{inequal4} that
\begin{equation*}
\log\mathbb{P}_{z}(|\tilde{Z}^{N}(\tau^{N}_{O})-y|<\delta, \tau^{N}_{O}<\theta_{\kappa_{1}})\geq -N\Big(V_{K_{0}^{c}}(z,y)+\eta\Big).
\end{equation*}
The lemma follows.
\end{proof}

\begin{lemma}\label{lemma2.5}
For all $\eta>0$, $y\in\widetilde{\partial O}$, there exists $\delta_{0}>0$ and $\rho_{0}>0$ (which can be chosen arbitrarily small) such that for any $\rho_{1}$, $0<\rho_{1}<\rho_{0}$, for any $\delta<\delta_{0}$ there exists $N_{0}$ so that for all $N>N_{0}$ and any $z\in G^{1}_{0}$,
\begin{equation*}
\mathbb{P}_{z}\left(\tau^{N}_{O}<\theta_{\kappa_{1}}; |\tilde{Z}^{N}(\tau^{N}_{O})-y|<\delta\right)\geq \exp(-N(V_{\bar{O}}(z^{\ast},y)+\eta)).
\end{equation*}
\end{lemma}

\begin{proof}
Let $\rho_{1}<\rho_{0}$, $\forall z\in G^{1}_{0}$  
\begin{align*}
\mathbb{P}_{z}\big(\tau^{N}_{O}<\theta_{\kappa_{1}}; |\tilde{Z}^{N}(\tau^{N}_{O})-y|<\delta\big)&= \mathbb{E}_{z}\Big(\mathbb{P}_{z}\big(\tau^{N}_{O}<\theta_{\kappa_{1}}; |\tilde{Z}^{N}(\tau^{N}_{O})-y|<\delta\big|\mathcal{F}_{\sigma_{0}}\big)\Big)\\
&=\mathbb{E}_{z}\Big(\mathbb{P}_{\tilde{Z}^{N}(\sigma_{0})}\big(\tau^{N}_{O}<\theta_{\kappa_{1}}; |\tilde{Z}^{N}(\tau^{N}_{O})-y|<\delta\big)\Big) \\
&\geqslant\inf_{v\in\Gamma_{0}}\mathbb{P}_{v}\big(\tau^{N}_{O}<\theta_{\kappa_{1}}; |\tilde{Z}^{N}(\tau^{N}_{O})-y|<\delta\big)
\end{align*}

 By definition of the $V_{\bar{O}}(z^{*},y)$, there exists a curve $\varphi(t)$, $0\leq t\leq T_{1}$  with $\varphi_{0}=z^{*}$, $\varphi_{T_{1}}=y$ such that
\begin{equation*}
  I_{T_{1}}(\varphi)\leq V_{\bar{O}}(z^{*},y)+\frac{\eta}{24}.
\end{equation*}
Moreover by Lemma \ref{assumption1} $v\in\Gamma_0$ and $z^{*}$ can be connected by a curve $\tilde{\varphi_{t}}$, $0\leq t\leq T_{2}$ such that 
\begin{equation*}
I_{T_{1}}(\tilde{\varphi})<\frac{\eta}{24}.
\end{equation*}
Combining $\tilde{\varphi_{t}}$, $0\leq t\leq T_{2}$ and $\varphi(t)$, $0\leq t\leq T_{1}$, we construct a curve $\bar{\varphi}$ by
\begin{equation*}
\bar{\varphi}_{t}=\begin{cases}
    \tilde{\varphi}_{t}  & \text{if }\quad t\in[0, T_{2}] , \\
      \varphi_{t-T_{2}} & \text{if}\quad t\in[T_{2}, T_{1}+T_{2}].
\end{cases}
\end{equation*}
so that 
\begin{equation*}
I_{T_{1}+T_{2}}(\bar{\varphi})\leq  V_{\bar{O}}(z^{\ast},y)+\frac{\eta}{12}
\end{equation*}
For this function $\bar{\varphi}$ we can choose $a$ such that its corresponding curve $\bar{\varphi}^{a}$, $0\leq t\leq T_{3}=T_{1}+T_{2}$ defined by 
\begin{equation*}
\bar{\varphi}^{a}_{t}=(1-a)\bar{\varphi}_{t}+a z_{0}
\end{equation*} 
is such that $\dist(\bar{\varphi}^{a}_{t},\widetilde{\partial{O}})\geq c_{2}a$ and
\begin{equation*}
  I_{T_{3}}(\bar{\varphi}^{a})\leq I_{T_{3}}(\bar{\varphi})+\frac{\eta}{12}.
\end{equation*}
Thus
\begin{equation*}
I_{T_{3}}(\bar{\varphi}^{a})\leq V_{\bar{O}}(z^{\ast},y)+\frac{\eta}{6}
\end{equation*}
According to the Lemma \ref{assumption1}, there exist two functions $\varphi^{a,1}_{t}$, $0\leq t\leq T^{a,1}$ and $\varphi^{a,2}_{t}$, $0\leq t\leq T^{a,2}$ such that
$\varphi^{a,1}_{0}=v$, $\varphi^{a,1}_{T^{a,1}}=\varphi^{a}_{0}$, $\varphi^{a,2}_{0}=\varphi^{a}_{T_{3}}$, $\varphi^{a,2}_{T^{a,2}}=y$ and
\begin{equation*}
I_{T^{a,1}}(\varphi^{a,1})\leq\frac{\eta}{12}\quad\text{and}\quad I_{T^{a,2}}(\varphi^{a,2})\leq\frac{\eta}{12}.
\end{equation*}
Concatenating $\varphi^{a,1}_{t}$, $0\leq t\leq T^{a,1}$, $\bar{\varphi}^{a}_{t-T^{a,1}}$, $T^{a,1}\leq t\leq T_{3}+T^{a,1}$ and $\phi^{a,2}_{t-T_{3}-T^{a,1}}$, $T_{3}+T^{a,1}\leq t\leq T^{a,2}+T_{3}+T^{a,1}$, we obtain a function $\psi_{t}$, $0\leq t\leq T$ such that
\begin{equation}\label{IPsi}
 I_{T}(\psi)\leq V_{\bar O}(z^{\ast},y)+\frac{\eta}{3}.
\end{equation}

 Let $\delta\leq\min\Big\{\frac{1}{2}c_{2} a,\frac{ \rho_0}{2}\Big\}$ and define $G\subseteq D_{T,\bar{O}}$ to be the set of functions $\phi$ having the following properties:
\begin{itemize}
  \item $|\phi_{T}-y|<\frac{1}{2}\delta$
  \item $\phi_{t}$ does not intersect $\big(\bar{O}\setminus O_{\frac{1}{2}\delta}\big) \setminus B_{\frac{1}{2}\delta}(y)$
  \item $\phi_{t}$ does not intersect $\overline{B_{\rho_{1}}(z^{*})}$
\end{itemize}
$G$ is open and $\psi\in G$ (if $\psi$ intersect $\overline{B_{\rho_{1}}(z^{*})}$ we use again Lemma \ref{assumption1} to modify it) we deduce from Theorem \ref{LDP} and \eqref{IPsi} that for  large enough $N$
\begin{equation*}
  \log\mathbb{P}^{N}_{v}(G)\geq -N\Big(\inf_{\phi\in G, \phi_{0}=v}I_{T}(\varphi)+\frac{\eta}{3}\Big)\ge
  -N\left(V_{\bar O}(z^{\ast},y)+\frac{2\eta}{3}\right).
\end{equation*}
Moreover we remark that $\tilde{Z}^{N}\in G$ implies $|\tilde{Z}^{N}(\tau^{N}_{O}\wedge T)-y|<\frac{1}{2}\delta$ and $\theta_{\kappa_{1}}>T$. So for $N$ large enough,
\begin{equation}\label{inequal5}
\log\mathbb{P}_{v}(|\tilde{Z}^{N}(\tau^{N}_{O}\wedge T)-y|<\frac{1}{2}\delta, \theta_{\kappa_{1}}>T)\geq -N\Big(V_{\bar O}(z^{\ast},y)+\frac{2\eta}{3}\Big).
\end{equation}
 We have moreover
\begin{align*}
\Theta&=\mathbb{P}_{v}\Big(|\tilde{Z}^{N}(\tau^{N}_{O})-y|<\delta, \theta_{\kappa_{1}}>\tau_{O}^{N}\Big)\\
&\geq \mathbb{P}_{v}\Big(|\tilde{Z}^{N}(\tau^{N}_{O}\wedge T)-y|<\frac{1}{2}\delta, |\tilde{Z}^{N}(\tau^{N}_{O})-y|<\delta, \theta_{\kappa_{1}}>\tau^{N}_{O}>T\Big) \\
&+\mathbb{P}_{v}\Big(|\tilde{Z}^{N}(\tau^{N}_{O}\wedge T)-y|<\frac{1}{2}\delta, |\tilde{Z}^{N}(\tau^{N}_{O})-y|<\delta,  \theta_{\kappa_{1}}>T\geq\tau^{N}_{O}\Big) \\
&+\mathbb{P}_{v}\Big(|\tilde{Z}^{N}(\tau^{N}_{O}\wedge T)-y|<\frac{1}{2}\delta, |\tilde{Z}^{N}(\tau^{N}_{O})-y|<\delta, T\geq\theta_{\kappa_{1}}>\tau^{N}_{O}\Big) \\
&\geq\Theta_{1}+\Theta_{2}+\Theta_{3}.
\end{align*}
Further
\begin{align*}
  \Theta_{1}&=\mathbb{E}_{v}\Big(\mathbb{P}_{v}\big(|\tilde{Z}^{N}(\tau^{N}_{O}\wedge T)-y|<\frac{1}{2}\delta, \theta_{\kappa_{1}}>\tau^{N}_{O}>T, |\tilde{Z}^{N}(\tau^{N}_{O})-y|<\delta|\mathcal{F}_{T}\big)\Big)\\
  &=\mathbb{E}_{v}\Big(\mathbb{P}_{v}\big(|\tilde{Z}^{N}(\tau^{N}_{O}\wedge T)-y|<\frac{1}{2}\delta, \theta_{\kappa_{1}}>\tau^{N}_{O}, \theta_{\kappa_{1}}\wedge\tau^{N}_{O}>T, |\tilde{Z}^{N}(\tau^{N}_{O})-y|<\delta|\mathcal{F}_{T}\big)\Big)\\
  &=\mathbb{E}_{v}\Big(|\tilde{Z}^{N}(T)-y|<\frac{1}{2}\delta, \theta_{\kappa_{1}}\wedge\tau^{N}_{O}>T, \mathbb{P}_{\tilde{Z}^{N}(T)}\big( \theta_{\kappa_{1}}>\tau^{N}_{O},  |\tilde{Z}^{N}(\tau^{N}_{O})-y|<\delta\big)\Big).
  \end{align*}

 Furthermore for $x$ such that $\dist(x,y)<\delta/2$ 
\begin{equation*}
  \mathbb{P}_{x}\Big(|\tilde{Z}^{N}(\tau^{N}_{O})-y|<\delta, \theta_{\kappa_{1}}>\tau_{O}^{N}\Big)\geq\mathbb{P}_{x}\Big(\tau_{O}^{N}<\gamma, \sup_{0\leq t\leq \gamma}|\tilde{Z}^{N}(t)-y|\leq\delta\Big),
\end{equation*}
for all $\gamma>0$.  In particular with $\gamma$ selected as in Lemma \ref{lemma2.3}
and $\delta=\rho$, we deduce that
\begin{equation*}
  \Theta_{1}\geq\mathbb{P}_{v}\Big(|\tilde{Z}^{N}(\tau^{N}_{O}\wedge T)-y|<\frac{1}{2}\delta, \theta_{\kappa_{1}}\wedge\tau^{N}_{O}>T\Big)\exp(-N\eta/3).
\end{equation*}
We also have
\begin{align*}
  \Theta_{2}&=\mathbb{P}_{v}\Big(|\tilde{Z}^{N}(\tau^{N}_{O}\wedge T)-y|<\frac{1}{2}\delta, \theta_{\kappa_{1}}>T\geq\tau^{N}_{O}\Big)  \\
  &\geq \mathbb{P}_{v}\Big(|\tilde{Z}^{N}(\tau^{N}_{O}\wedge T)-y|<\frac{1}{2}\delta, \theta_{\kappa_{1}}>T\geq\tau^{N}_{O}\Big)\exp(-N\eta/3).
\end{align*}
Thus since $\Theta_{3}\ge0$,
\begin{align}\label{inequal6}
  \Theta&\geq\mathbb{P}_{v}\Big(|\tilde{Z}^{N}(\tau^{N}_{O}\wedge T)-y|<\frac{1}{2}\delta, \theta_{\kappa_{1}}\wedge\tau^{N}_{O}>T\Big) \exp(-N\eta/3) \nonumber\\
  &+\mathbb{P}_{v}\Big(|\tilde{Z}^{N}(\tau^{N}_{O}\wedge T)-y|<\frac{1}{2}\delta, \theta_{\kappa_{1}}>T\geq\tau^{N}_{O}\Big) \exp(-N\eta/3)  \nonumber\\
  &=\mathbb{P}_{v}\Big(|\tilde{Z}^{N}(\tau^{N}_{O}\wedge T)-y|<\frac{1}{2}\delta, \theta_{\kappa_{1}}>T\Big)\exp(-N\eta/3).
\end{align}
It follows from \eqref{inequal5} and \eqref{inequal6} that

\begin{equation*}
\log\mathbb{P}_{v}(|\tilde{Z}^{N}(\tau^{N}_{O})-y|<\delta, \theta_{\kappa_{1}}>\tau^{N}_{O})\geq -N\Big(V_{\bar O}(z^{\ast},y)+\eta\Big).
\end{equation*}
\end{proof}

\begin{lemma}\label{lemma2.6}
Given any $\eta$, there exists $\rho_{0}>0$ (which can be chosen arbitrarily small) such that for any $\rho_{2}\in(0,\rho_{0})$, there exists $\rho_{1}\in(0, \rho_{2})$  and $N_{\eta}$ such that for all $N>N_{\eta}$ and $z\in G^{1}_{0}$,
\begin{equation*}
\mathbb{P}_{z}\left(\tau^{N}_{O}<\theta_{\kappa_{1}}\right)\leq \exp(-N(V_{\widetilde{\partial O}}-\eta)).
\end{equation*}
\end{lemma}

\begin{proof}
Let $\delta>0$, we define
\begin{equation*}
  \tau^{N}_{O_{\delta}}=\inf\{t>0: \tilde{Z}^{N,z}(t)\not\in O_{\delta}\}.
\end{equation*}
It is easy to see that $\tau^{N}_{O_{\delta}}<\tau^{N}_{O}$. Moreover if $\rho_{1}<\delta$, then $\tau^{N}_{O_{\delta}}<\theta_{\kappa_{1}}$ implies $\tau^{N}_{O_{\delta}}<\theta_{1}$. Thus $\forall z\in G^{1}_{0}$
\begin{equation*}
  \mathbb{P}_{z}(\tau^{N}_{O}<\theta_{\kappa_{1}})\leq\mathbb{P}_{z}(\tau^{N}_{O_{\delta}}<\theta_{\kappa_{1}})\leq\mathbb{P}_{z}(\tau^{N}_{O_{\delta}}<\theta_{1}).
\end{equation*}
Now we use the strong Markov property to write that $\forall z\in G^{1}_{0}$
\begin{equation*}
  \mathbb{P}_{z}(\tau^{N}_{O_{\delta}}<\theta_{1})=\mathbb{E}_{z}\Big(\mathbb{P}_{\tilde{Z}^{N}(\sigma_{0})}(\tau^{N}_{O_{\delta}}<\theta_{1})\Big)
\end{equation*}
we deduce that
\begin{equation*}
  \sup_{z\in G^{1}_{0}}\mathbb{P}_{z}(\tau^{N}_{O}<\theta_{\kappa_{1}})\leqslant\sup_{v\in\Gamma_{0}}\mathbb{P}_{v}(\tau^{N}_{O_{\delta}}<\theta_{1})
\end{equation*}
Now, we establish that we can choose $\rho_{0}$ and $\delta$ sufficiently small such that for all $v\in\Gamma_{0}$ we have,
\begin{equation*}
\mathbb{P}_{v}(\tau^{N}_{O_{\delta}}<\theta_{1})\leq\exp\{-N(V_{\widetilde{\partial O_{\delta}}}-\frac{2\eta}{3})\}.
\end{equation*}
Using Lemma \ref{assumption1} there exists $\rho>0$ such that $T(\rho)<\eta/3\mathcal{K}$. We take $\rho_{0}<\rho/2$, $\delta$ and $\gamma$ sufficiently small such that for any trajectory $\phi_{t}$, $0\leq t\leq T$ starting from $v\in\Gamma_{0}$ and touching $O_{\delta}\setminus O_{\delta+\gamma}$ we have
\begin{equation*}
I_{T}(\phi)\geq V_{\widetilde{\partial O_{\delta}}}-\frac{\eta}{4}.
\end{equation*}
Moreover, using Lemma \ref{lemecomp} there exists a constant $c$ and $T_{1}$ such that for all $N$ large enough, any $T>T_{1}$ and $v\in \overline{O_{\delta}\setminus G}$ where $G=\bigcup_{i=1}^{M}G^{1}_{i}$ we have
\begin{equation*}
\mathbb{P}_{v}(\tau^{N}_{\overline{O_{\delta}\setminus G}}>T)\leq\exp\{-N c(T-T_{1})\}.
\end{equation*}
Now if we take any trajectory $\tilde{Z}^{N,v}$, with $v\in\Gamma_{0}$, and which reaches $O\setminus O_{\delta}$ before going to $G$ either spends time $T$ without touching $O\setminus O_{\delta}$ (the event $\{\tau^{N}_{\overline{O_{\delta}\setminus G}}>T\}$ is realized) or reaches $O\setminus O_{\delta}$ before time $T$ and in this case the event  
\begin{equation*}
\Big\{\rho_{T}\Big(\tilde{Z}^{N,v},\Phi(V_{\widetilde{\partial O_{\delta}}}-\frac{\eta}{4})\-\Big)\geq \gamma\Big\} \text{ is realized}. 
\end{equation*}
Hence, for all $v\in\Gamma_{0}$ we have from Lemma \ref{lemecomp} and Theorem \ref{LDP} (c)
\begin{align*}
    \mathbb{P}_{v}(\tau^{N}_{O_{\delta}}<\theta_{1})&\leq \mathbb{P}_{v}(\tau^{N}_{\overline{O_{\delta}\setminus G}}>T)+ \mathbb{P}_{v}\big(\big\{\rho_{T}\big( \tilde{Z}^{N,v},\Phi(V_{\widetilde{\partial O_{\delta}}}-\frac{\eta}{4})\big)\geq \gamma\big\}\big) \\
    &\leq \exp\{-N c(T-T_{1})\}+\exp\{-N(V_{\widetilde{\partial O_{\delta}}}-\frac{\eta}{2})\} \\
    &\leq 2\exp\{-N(V_{\widetilde{\partial O_{\delta}}}-\frac{\eta}{2})\} \quad\text{taking $T$ large enough} \\
    &\leq \exp\{-N(V_{\widetilde{\partial O_{\delta}}}-\frac{3\eta}{4})\}  \quad\text{taking $N$ large enough such that $\frac{\ln(2)}{N}<\frac{\eta}{4}$}
\end{align*}
Moreover, $V_{\bar{O}}(z^{\ast},.)$ is continuous so if $\delta$ is sufficiently small then
\begin{equation*}
V_{\widetilde{\partial O_{\delta}}}\geq V_{\widetilde{\partial O}}-\frac{\eta}{4}.
\end{equation*}
So 
\begin{equation*}
\sup_{z\in G^{1}_{0}}\mathbb{P}_{z}(\tau^{N}_{O}<\theta_{\kappa_{1}})\leq\exp\{-N(V_{\widetilde{\partial O}}-\eta)\}.
\end{equation*}
This prove the lemma. 
\end{proof}

\section{Main Result on the Exit Position}\label{exitposition}

 Before to formulate the main result of our paper (Theorem \ref{mainresult1}), we introduce here a notion important  to understand the proof of that result.
\begin{definition}
Let $\mathcal{L}$ a subset of $\mathbb{N}$ and $W$ a subset of $\mathcal{L}$. A W-graph on $\mathcal{L}$ is an oriented graph which satisfies the following conditions
\begin{description}
  \item[(a)] 
  It consists of arrows $i\to j$, $i\neq j$ with $i\in \mathcal{L}\setminus W$ and $j\in \mathcal{L}$.
  \item[(b)]
  For all $i\in \mathcal{L}\setminus W$, there exists exactly one arrow such that $i$ is its initial point.
   \item[(c)] 
  For any $i\in \mathcal{L}\setminus W$ there exists a sequence of arrows leading from it to some point $j\in W$.
  \end{description}
\end{definition}

 We will denote by $Gr(W)$ the set of $W$-graphs and for $i\in \mathcal{L}\setminus W$, $j\in W$  we denote by $Gr_{i,j}(W)$ the set of $W$-graphs in which the sequence of arrows leading from $i$ into $W$ ends at $j$.
We can now show the analogue of Lemma 3.2 in \cite{day1990large}.

\begin{lemma}\label{lem6}
For all $y\in\widetilde{\partial O}$ and $\eta$, $\delta_{0}>0$ there exist $\rho$, $\delta<\delta_{0}$ and $N_{0}$, so that for all $z$ with $\dist(z,z^{*})<\rho$ and $N>N_{0}$, we have
\begin{equation}\label{lower2}
\exp(-N(V_{\bar{O}}(z^{*},y)-V_{\widetilde{\partial O}}+\eta))\leq\mathbb{P}_{z}(|\tilde{Z}^{N}(\tau^{N}_{O})-y|<\delta)
\end{equation}
and
\begin{equation}\label{upper2}
\mathbb{P}_{z}(|\tilde{Z}^{N}(\tau^{N}_{O})-y|<\delta)\leq\exp(-N(V_{\bar{O}}(z^{*},y)-V_{\widetilde{\partial O}}-\eta))
\end{equation}
\end{lemma}

\begin{proof}
Let $y\in\widetilde{\partial O}$. We can always assume that $y\in\bigcup_{1}^{M}K_{i}$ else we add the compact $K=\{y\}$ in the list of the compacts and Assumption \ref{assump2} remains true since $V_{O}(y,y)=0$ and If $y\mathcal{R} u$ for some $u\neq y$, then Lemma \ref{lemequiva} implies that any $\omega$-limit point of \eqref{ODE} starting at $y$ is equivalent to $y$ and then $y$ was in a compact  $K_{i}$. 

 In what follows we assume that $y\in K_{1}$.
Let $\delta>0$. Using the strong Markov property we have, for all $z\in G^{1}_{0}$,
\begin{align}
  &\mathbb{P}_{z}(|\tilde{Z}^{N}(\tau^{N}_{O})-y|<\delta) \nonumber\\
  &= \sum_{k=0}^{\infty}\mathbb{P}_{z}(|\tilde{Z}^{N}(\tau^{N}_{O})-y|<\delta; \theta_{\kappa_{k}}\leq\tau^{N}_{O}<\theta_{\kappa_{k+1}})  \nonumber\\
  &= \sum_{k=0}^{\infty}\mathbb{E}_{z}\Big[\mathbb{P}_{z}(|\tilde{Z}^{N}(\tau^{N}_{O})-y|<\delta; \theta_{\kappa_{k}}\leq\tau^{N}_{O}<\theta_{\kappa_{k+1}}|\mathcal{F}_{\theta_{\kappa_{k}}})\Big]  \nonumber\\
  &= \sum_{k=0}^{\infty}\mathbb{E}_{z}\Big[\mathbb{P}_{z}(|\tilde{Z}^{N}(\tau^{N}_{O})-y|<\delta; \theta_{\kappa_{k}}\leq\tau^{N}_{O}<\theta_{\kappa_{k+1}}|Z^{N}(\theta_{\kappa_{k}}))\Big]  \nonumber\\   
  &= \sum_{k=0}^{\infty}\mathbb{E}_{z}\Big[\mathbb{P}_{\tilde{Z}^{N}(\theta_{\kappa_{k}})}(|\tilde{Z}^{N}(\tau^{N}_{O})-y|< \delta; \tau^{N}_{O}<\theta_{\kappa_{1}}); \theta_{\kappa_{k}}\leq\tau^{N}_{O}\Big]  \nonumber\\
   &= \sum_{k=0}^{\infty}\mathbb{E}_{z}\Big[\mathbb{P}_{\tilde{Z}^{N}(\theta_{\kappa_{k}})}(|\tilde{Z}^{N}(\tau^{N}_{O})-y|< \delta| \tau^{N}_{O}<\theta_{\kappa_{1}})\mathbb{P}_{\tilde{Z}^{N}(\theta_{\kappa_{k}})}(\tau^{N}_{O}<\theta_{\kappa_{1}}); \theta_{\kappa_{k}}\leq\tau^{N}_{O}\Big] \label{maineq}
\end{align}
Furthermore, $\forall z\in G^{1}_{0}$ we have
\begin{equation}\label{pro1}
\mathbb{P}_{z}(|\tilde{Z}^{N}(\tau^{N}_{O})-y|< \delta| \tau^{N}_{O}<\theta_{\kappa_{1}}) =\frac{\mathbb{P}_{z}(|\tilde{Z}^{N}(\tau^{N}_{O})-y|< \delta; \tau^{N}_{O}<\theta_{\kappa_{1}})}{\mathbb{P}_{z}( \tau^{N}_{O}<\theta_{\kappa_{1}})}
\end{equation}
Let us now deduce the lower bound \eqref{lower2} from \eqref{pro1}, \eqref{maineq} combined with Lemmas \ref{lemma2.5} and \ref{lemma2.6}.  Given $\eta$, $\delta>0$ pick $0<\rho_{1}<\rho_{0}$ satisfying both Lemmas with $\eta$ replaced by $\eta/2$.  For all $z\in G^{1}_{0}$  and  $N$ large enough, we have
\begin{align*}
  \frac{\mathbb{P}_{z}(|\tilde{Z}^{N}(\tau^{N}_{O})-y|< \delta; \tau^{N}_{O}<\theta_{\kappa_{1}})}{\mathbb{P}_{z}(\tau^{N}_{O}<\theta_{\kappa_{1}})} &\geq \frac{\exp(-N(V_{\bar{O}}(z^{*},y)+\eta/2))}{\exp(-N(V_{\tilde{\partial O}}-\eta/2))} \\
   &= \exp(-N(V_{\bar{O}}(z^{*},y)-V_{\tilde{\partial O}}+\eta))
\end{align*}
The lower bound \eqref{lower2} follows from this and \eqref{maineq}.

 For the upper bound \eqref{upper2}, we first obtain a lower bound of the denominator of \eqref{pro1} as follows: \begin{align*}
\mathbb{P}_{z}(\tau^{N}_{O}<\theta_{\kappa_{1}}) &= \mathbb{P}_{z}\left(\tilde{Z}_{1}\in \bigcup_{\ell=1}^{M} G^{1}_{\ell}; \tilde{Z}^{N}(t^-)+\frac{1}{N}\sum_{j=1}^{k}h_{j}\triangle Q_{j}^{N}(t)\not\in O\quad\text{for some}\quad t\in [\theta_{1},\theta_{2})\right) \\
   &= \mathbb{E}_{z}\left(\mathbb{P}_{\tilde{Z}_{1}}(\tau^{N}_{O}<\theta_{1}); \tilde{Z}_{1}\in \bigcup_{\ell=1}^{M} G^{1}_{\ell}\right)
\end{align*}
Now we use Lemma \ref{lemma2.3} to deduce that choosing $\gamma,\rho_{1}$ such that $0<\gamma<\eta$,  $0<\rho_{1}<\eta$ and $N_{0}\in\mathbb{N}$ we have for all $u\in\bigcup_{1}^{M} G^{1}_{\ell}$, $\dist(u,\widetilde{\partial O})<\rho_{1}/2$ and then for $N>N_{0}$,
\begin{align*}
  \mathbb{P}_{u}(\tau^{N}_{O}<\theta_{1})&\geq \mathbb{P}_{u}\Big(\tau_{O}^{N}<\gamma; \sup_{0\leq t\leq \gamma}|\tilde{Z}^{N}(t)-u|<\rho_{1}\Big) \\
  &\geq\exp(-N\eta).
\end{align*}
 Thus, for all $\rho_{1}$ sufficiently small, $N$ large enough and all $z\in G^{1}_{0}$,
\begin{equation}\label{pro2}
  \mathbb{P}_{z}(\tau^{N}_{O}<\theta_{\kappa_{1}})\geq\exp(-N\eta/4)\mathbb{P}_{z}\left( \tilde{Z}_{1}\in \bigcup_{1}^{M} G^{1}_{\ell}\right)
\end{equation}
By using the lemma \ref{lemma2.2} we have for all suitable small $\rho_{1}$, $\rho_{0}$, and $N$ large enough, and all $z\in G^{1}_{0}$,
\begin{equation}\label{pro3}
  \mathbb{P}_{z}\left( \tilde{Z}_{1}\in \bigcup_{\ell=1}^{M} G^{1}_{\ell}\right)\geq\exp\{-N(V_{\tilde{\partial O}}+\eta/4)\}
\end{equation}
We then deduce of \eqref{pro2} and \eqref{pro3} that
\begin{equation}\label{pro4}
  \mathbb{P}_{z}(\tau^{N}_{O}<\theta_{\kappa_{1}})\geq\exp(-N(V_{\tilde{\partial O}}+\eta/2)).
\end{equation}
 We now use the embedded chain $\tilde{Z}_{n}$ to obtain a upper bound of the numerator of \eqref{pro1} in the following way: given $\delta<\rho_{1}$ and $z\in G^{1}_{0}$,
\begin{align}
  \mathbb{P}_{z}(|\tilde{Z}^{N}(\tau^{N}_{O})-y|< \delta; \tau^{N}_{O}<\theta_{\kappa_{1}}) &\leq \mathbb{P}_{z}(\tilde{Z}^{N}(\tau^{N}_{O})\in G^{1}_{1}; \tau^{N}_{O}<\theta_{\kappa_{1}}) \nonumber \\
   &\leq \mathbb{P}_{z}(\tilde{Z}_{n}\in G^{1}_{1}\quad\text{for some}\quad 1\leq n<\kappa_{1})  \label{pro5}  \\
   &= \mathbb{E}_{z}(\mathbb{P}_{\tilde{Z}_{1}}(\tilde{Z}_{n}\in G^{1}_{1}\quad\text{for some}\quad 0\leq  n<\kappa_{1})), \label{pro6} 
\end{align}
where we assume of course that
\begin{equation*}
\mathbb{P}_{v}\big(\tilde{Z}_{n}\in G^{1}_{1}\quad\text{for some}\quad 0\leq  n<\kappa_{1}\big):=
\begin{cases}
1& \text{ if }\quad v\in G^{1}_{1}, \\
0 & \text{if}\quad v\in G^{1}_{0}.
\end{cases}
\end{equation*}

 Now we try to have an upper bound of $\mathbb{P}_{v}(\tilde{Z}_{n}\in G^{1}_{1}\quad\text{for some}\quad 0\leq  n<\kappa_{1})$ for $v\in G^{1}_{\ell}$ where $\ell\neq0,1$.   For all suitable $\rho_{1}$, $\rho_{0}$ and $N$ large we have for all $v\in G^{1}_{\ell}$
\begin{equation*}
  \mathbb{P}_{v}(\tilde{Z}_{n}\in G^{1}_{1}\quad\text{for some}\quad 0\leq  n<\kappa_{1})= q_{W}(v,G^{1}_{1}).
\end{equation*}
Where $q_{W}(v,G^{1}_{1})$ is the probability that, starting from $v$ the Markov chain $(\tilde{Z}_{n})$ hits $G^{1}_{1}$ when it first enters $G^{1}_{0}\cup G^{1}_{1}$.
 
 Now we will use a result on the Markov chains described by \cite{freidlin2012random} lemma 3.3 of chapter 6 in terms of the $W$-graphs on the sets $\mathcal{L}=\{0,\dots,M\}$ where $W=\{0, 1\}$. To apply this lemma we define the sets $X_{i}=G^{1}_{i}$ $\forall i\in\{0,\dots,M\}$ and $X=\underset{i=0}{\overset{M}{\bigcup}}X_{i}$. If we define $a=\exp\{N\eta/4^{M-1}\}$  and $p_{i,j}=\exp\{-N\widetilde{V}_{\bar{O}}(K_{i},K_{j})\}$,  we deduce from lemma \ref{esti1} that the assumptions of the lemma 3.3 of chapter 6 in \cite{freidlin2012random} hold true, hence for all suitable $\rho_{1}$, $\rho_{0}$ and $N$ large enough, $v\in G^1_\ell$,
\begin{equation*}
q_{W}(v,G^{1}_{1})\leq a^{4^{M-1}}\frac{\sum_{g\in Gr_{\ell,1}(W)}\prod_{(i\rightarrow j)\in g} p_{i,j}}{\sum_{g\in Gr(W)}\prod_{(i\rightarrow j)\in g} p_{i,j}} .
\end{equation*}
 Thus
 \begin{equation}\label{ineqq}
q_{W}(v,G^{1}_{1})\leq\exp\{N \eta/7\}\frac{\sum_{g\in Gr_{\ell,1}(W)}\exp\Big\{-N\sum_{(i\rightarrow j)\in g} \widetilde{V}_{\bar{O}}(K_{i},K_{j})\Big\}}{\sum_{g\in Gr(W)}\exp\Big\{-N\sum_{(i\rightarrow j)\in g}\widetilde{V}_{\bar{O}}(K_{i},K_{j})\Big\} } .
\end{equation}
It is easy to see that $\sum_{g\in Gr_{\ell,1}(W)}\exp\Big\{-N\sum_{(i\rightarrow j)\in g} \widetilde{V}_{\bar{O}}(K_{i},K_{j})\Big\}$ is equivalent to a positive constant $C_{1}$ which is the number of graphs in $Gr_{\ell,1}(W)$ at which the minimum of $\sum_{(i\rightarrow j)\in g} \widetilde{V}_{\bar{O}}(K_{i},K_{j})$ is attained multiplied by $\exp\Big\{-N\min_{g\in Gr_{\ell,1}(W)}\sum_{(i\rightarrow j)\in g} \widetilde{V}_{\bar{O}}(K_{i},K_{j})\Big\}$. We also see easily that the denominator in \eqref{ineqq} is equivalent to a positive constant multiplied by  $\exp\Big\{-N\min_{g\in Gr(W)}\sum_{(i\rightarrow j)\in g} \widetilde{V}_{\bar{O}}(K_{i},K_{j})\Big\}$. Hence there exists $N_0$ such that for $N\ge N_0$
\begin{align*}
q_{W}(v,G^{1}_{1}) 
&\leq \exp\Big\{-N\left(\min_{g\in Gr_{\ell,1}(W)}\sum_{(i,j)\in g}\widetilde{V}_{\bar{O}}(K_{i},K_{j})-\min_{g\in Gr(W)}\sum_{(i,j)\in g}\widetilde{V}_{\bar{O}}(K_{i},K_{j})\right)+N\eta/6\Big\}.
\end{align*}
 
 We remark here that in the case of a single attracting set $K_{0}$, $\widetilde{V}_{\bar{O}}(K_{i},K_{0})=0$ for all $i$. Then we have
\begin{equation*}
  \min_{g\in Gr(W)}\sum_{(i,j)\in g}\widetilde{V}_{\bar{O}}(K_{i},K_{j})=0
\end{equation*}
We also have
\begin{equation*}
  \min_{g\in Gr_{\ell,1}(W)}\sum_{(i,j)\in g}\widetilde{V}_{\bar{O}}(K_{i},K_{j})=V_{\bar{O},K_{0}^{c}}(K_{\ell},K_{1}).
\end{equation*}
With these preceding remark, we deduce that
\begin{equation*}
  \mathbb{P}_{v}(\tilde{Z}_{n}\in G^{1}_{1}\quad \text{for some}\quad0\leq  n<\kappa_{1})\leq\exp\{-N(V_{\bar{O},K_{0}^{c}}(K_{\ell},K_{1})-\eta/6)\};\quad v\in G^{1}_{\ell},\quad\ell\neq0,1.
\end{equation*}
Now according \eqref{pro6} and lemma \ref{esti1} we have for $N>\frac{6\log(M)}{\eta}\vee N_0$ \begin{align*}
  &\mathbb{P}_{z}(|\tilde{Z}^{N}(\tau^{N}_{O})-y|< \delta; \tau^{N}_{O}<\theta_{\kappa_{1}})\leq \mathbb{E}_{z}\Big(\sum_{\ell=0}^{M}\mathds{1}_{Z_{1}\in G^{1}_{\ell}}\mathbb{P}_{Z_{1}}\big(Z_{n}\in G^{1}_{1}\quad\text{for}\quad\text{some}\quad 0\leq  n<\kappa_{1}\big)\Big)\\
  &=\mathbb{P}_{z}(\tilde{Z}_{1}\in G^{1}_{1})+\sum_{\ell=2}^{M}\mathbb{E}_{z}\Big(\mathds{1}_{\tilde{Z}_{1}\in G^{1}_{\ell}}\mathbb{P}_{\tilde{Z}_{1}}\big(\tilde{Z}_{n}\in G^{1}_{1}\quad\text{for}\quad\text{some}\quad 0\leq  n<\kappa_{1}\big)\Big)\\
  &\leq \exp\Big\{-N(\widetilde{V}_{\bar{O}}(K_{0},K_{1})-\eta/6)\}+\sum_{\ell=2}^{M} \exp\{-N(\widetilde{V}_{\bar{O}}(K_{0},K_{\ell})+V_{\bar{O},K_{0}^{c}}(K_{\ell},K_{1})-\eta/6)\Big\}\\
  & \leq\exp\Big\{-N(\widetilde{V}_{\bar{O}}(K_{0},K_{1})-\eta/6)\}+\exp\{-N(\min_{2\leq\ell\leq M}\{\widetilde{V}_{\bar{O}}(K_{0},K_{\ell})+ V_{\bar{O},K_{0}^{c}}(K_{\ell},K_{1})\}-\eta/3)\Big\} \\
  &\leq \exp\Big\{-N\big(\widetilde{V}_{\bar{O}}(K_{0},K_{1})\wedge\min_{2\leq\ell\leq M}\{\widetilde{V}_{\bar{O}}(K_{0},K_{\ell})+ V_{\bar{O},K_{0}^{c}}(K_{\ell},K_{1})\}-\eta/2\big)\Big\}.
\end{align*}
We remark here that 
\begin{align*}
  &\widetilde{V}_{\bar{O}}(K_{0},K_{1})\wedge\min_{2\leq\ell\leq M}\{\widetilde{V}_{\bar{O}}(K_{0},K_{\ell})+ V_{\bar{O},K_{0}^{c}}(K_{\ell},K_{1})\} \\
  &= V_{\bar{O}}(K_{0},K_{1}) \\
   &= V_{\bar{O}}(z^{\ast},y)
\end{align*}
And then
\begin{equation}\label{equation1}
 \mathbb{P}_{z}(|\tilde{Z}^{N}(\tau^{N}_{O})-y|< \delta; \tau^{N}_{O}<\theta_{\kappa_{1}})\leq\exp\{-N(V(z^{\ast},y)-\eta/2)\}
\end{equation}
uniformly over $z\in G^{1}_{0}$, provided $\rho_{1}$, $\rho_{0}$ and $N$ are chosen in suitable way.
\par Combining \eqref{pro4}, \eqref{equation1} and \eqref{pro1}, we deduce that
\begin{equation*}
\mathbb{P}_{z}(|\tilde{Z}^{N}(\tau^{N}_{O})-y|< \delta| \tau^{N}_{O}<\theta_{\kappa_{1}})\leq\exp\{-N(V_{\bar{O}}(z^{\ast},y)-V_{\tilde{\partial O}}-\eta)\},\quad\text{for all $z\in G^{1}_{0}$},
\end{equation*}
provided $\delta<\rho_{1}$ with $\rho_{1}$ sufficiently small and $N$ sufficiently large. As $\eta>0$ is arbitrary, we obtain the upper bound \eqref{upper2} from \eqref{maineq}. This concludes the proof of lemma \ref{lem6}.
\end{proof}

 We finally deduce our main result which is an analogue of Theorem 3.1 in \cite{day1990large}.
\begin{theorem}\label{mainresult1}
For  $z\in O$, $y\in\widetilde{\partial O}$ and any  $\eta$, $\delta_{0}>0$ there exist $\delta<\delta_{0}$ and $N_{0}$, such that for all $N>N_{0}$
\begin{equation}\label{lower13}
\exp(-N(S_{z}(y)+\eta))\leq\mathbb{P}_{z}(|\tilde{Z}^{N}(\tau^{N}_{O})-y|<\delta)
\end{equation}
and
\begin{equation}\label{upper13}
\mathbb{P}_{z}(|\tilde{Z}^{N}(\tau^{N}_{O})-y|<\delta)\leq\exp(-N(S_{z}(y)-\eta))
\end{equation}
where $S_{z}(y)$ is defined by:
\begin{equation}\nonumber
S_{z}(y)=V_{\bar{O}}(z,y)\wedge(V_{\bar{O}}(z^{\ast},y)-V_{\widetilde{\partial O}}).
\end{equation}
\end{theorem}

\begin{proof}
 We first remark that for $z=z^{\ast}$ the result is given by Lemma \ref{lem6} . If $z\in O\setminus\{z^{\ast}\}$, we make the restriction that $\rho_{0}$ be sufficiently small so that $\dist(z,K_{i})>\rho_{0}$ for all $i=0,...,M$. This allows us to write
\begin{align}\label{ineqI}
     \mathbb{P}_{z}(|\tilde{Z}^{N}(\tau^{N}_{O})-y|<\delta)&=  \mathbb{P}_{z}(|\tilde{Z}^{N}(\tau^{N}_{O})-y|<\delta, \theta_{\kappa_{1}}<\tau^{N}_{O}) \nonumber\\
    &+ \mathbb{P}_{z}(|\tilde{Z}^{N}(\tau^{N}_{O})-y|<\delta, \theta_{\kappa_{1}}>\tau^{N}_{O}))  \nonumber\\
    &= \mathbb{E}_{z}(\mathbb{P}_{\tilde{Z}_{\kappa_{1}}}(|\tilde{Z}^{N}(\tau^{N}_{O})-y|<\delta); \theta_{\kappa_{1}}<\tau^{N}_{O}) \nonumber\\ 
    &+\mathbb{P}_{z}(|\tilde{Z}^{N}(\tau^{N}_{O})-y|<\delta, \theta_{\kappa_{1}}>\tau^{N}_{O}) \nonumber \\
    &= \Pi_{1}+\Pi_{2}. 
    \end{align}
\begin{description}
  \item[Upper bound of $\Pi_{1}$.] 
Lemma \ref{lem6} tells us that for all $z\in G_{0}^{1}$,
\begin{equation*}
\mathbb{P}_{z}(|\tilde{Z}^{N}(\tau^{N}_{O})-y|<\delta) \leq\exp\{-N(V_{\bar{O}}(z^{\ast},y)-V_{\widetilde{\partial O}}-\eta/3)\}.
\end{equation*}
Hence $\Pi_{1}$ can be bounded from above as follows
\begin{align*}
     \Pi_{1}&= \mathbb{E}_{z}(\mathbb{P}_{\tilde{Z}_{\kappa_{1}}}\{|\tilde{Z}^{N}(\tau^{N}_{O})-y|<\delta\} ;\theta_{\kappa_{1}}<\tau^{N}_{O}) \\
     &\leq\mathbb{P}_{z}\{\theta_{\kappa_{1}}<\tau^{N}_{O})\}\exp\{-N(V_{\bar{O}}(z^{\ast},y)-V_{\widetilde{\partial O}}-\eta/3)\}   \\
    &\leq\exp\{-N(V_{\bar{O}}(z^{\ast},y)-V_{\widetilde{\partial O}}-\eta/3)\}. 
\end{align*}

  \item[Lower bound of $\Pi_{1}$.]
  From Lemma \ref{Limit} we deduce that for suitably small $\rho_{0}, \rho_{1}$ there exists $N\in \mathbb{N}$ such that, for all $N>N_{0}$,
  
\begin{align*}
    \mathbb{P}_{z}\{\theta_{\kappa_{1}}<\tau^{N}_{O}\}&\geq\mathbb{P}_{z}\{\tilde{Z}^{N}(\theta_{\kappa_{1}}\wedge\tau_{O}^{N})\in G^{1}_{0}\}  \\ 
    &\geq\exp\{-N\eta/2\}.
\end{align*}
  
  $\Pi_{1}$ can then be bounded from below as follows using Lemma \ref{lem6}
\begin{align*}
      \Pi_{1}&=\mathbb{E}_{z}(\mathbb{P}_{\tilde{Z}_{\kappa_{1}}}(|\tilde{Z}^{N}(\tau^{N}_{O})-y|<\delta) ,\theta_{\kappa_{1}}<\tau^{N}_{O}) \\
      &\geq \mathbb{P}_{z}( \theta_{\kappa_{1}}<\tau^{N}_{O})\exp\{-N(V_{\bar{O}}(z^{*},y)-V_{\widetilde{\partial O}}+\eta/2)\}   \\
    & \geq\exp\{-N(V_{\bar{O}}(z^{*},y)-V_{\widetilde{\partial O}}+\eta)\}.
\end{align*}

  \item[Upper bound of $\Pi_{2}$.]
 We now obtain a upper bound for $\Pi_{2}$ by making the same computations as in the proof of lemma \ref{lem6}. Indeed
\begin{align*}
  \mathbb{P}_{z}(|\tilde{Z}^{N}(\tau^{N}_{O})-y|< \delta; \tau^{N}_{O}<\theta_{\kappa_{1}}) &\leq \mathbb{P}_{z}(\tilde{Z}^{N}(\tau^{N}_{O})\in G^{1}_{1}; \tau^{N}_{O}<\theta_{\kappa_{1}}) \\
  &\leq \mathbb{P}_{z}(\tilde{Z}_{n}\in G^{1}_{1}\quad\text{for some}\quad 0\leq n\leq\kappa_{1}) \\
   &=\mathbb{E}_{z}(\mathbb{P}_{\tilde{Z}_{1}}(\tilde{Z}_{n}\in G^{1}_{1}\quad\text{for some}\quad 0\leq  n<\kappa_{1})).
\end{align*}
Where we assume of course that,
\begin{equation*}
\mathbb{P}_{v}(\tilde{Z}_{n}\in G^{1}_{1}\quad\text{for some}\quad 0\leq  n<\kappa_{1}):=
\begin{cases}
1& \text{ if }\quad v\in G^{1}_{1} \\
0 & \text{if}\quad v\in G^{1}_{0}.
\end{cases}
\end{equation*}
For $v\in G_{\ell}^{1}$, $\ell\neq0,1$ we can establish as in the proof of Lemma \ref{lem6} that
\begin{equation*}
\mathbb{P}_{v}(\tilde{Z}_{n}\in G^{1}_{1}\quad\text{for some}\quad 0\leq  n<\kappa_{1})\leqslant \exp\{-N(V_{\bar{O},K_{0}^{c}}(K_{\ell},K_{1})-\eta)\}
\end{equation*} 
Hence 
\begin{align*}
     \Pi_{2}&= \mathbb{P}_{z}(|\tilde{Z}^{N}(\tau^{N}_{O})-y|<\delta, \tau^{N}_{O}<\theta_{\kappa_{1}})  \\
    &  \leq \exp\{-N(\widetilde{V}_{\bar{O}}(z,K_{1})\wedge\min_{2\leq \ell\leq M}[\widetilde{V}_{\bar{O}}(z,K_{\ell})+V_{\bar{O},K_{0}^{c}}(K_{\ell},K_{1})]-2\eta/3)\} \\
    & \leq \exp\{-N(V_{\bar{O},K_{0}^{c}}(z,y)-2\eta/3)\}.
\end{align*}

  \item[Lower bound of $\Pi_{2}$.]
   From Lemma \ref{lemma2.4}, we deduce
\begin{align*}
\Pi_{2}&=\mathbb{P}_{z}(|\tilde{Z}^{N}(\tau^{N}_{O})-y|<\delta, \theta_{\kappa_{1}}>\tau^{N}_{O}) \\
&\geq\exp\{-N(V_{\bar{O},K_{0}^{c}}(z,y)+\eta)\}.
\end{align*}

\item[Conclusion]
Thus, the term on the left in \eqref{ineqI} can be bounded from above as follows, provided that $N>\frac{3\log(2)}{\eta}$
\begin{align*}
    &  \mathbb{P}_{z}(|\tilde{Z}^{N}(\tau^{N}_{O})-y|<\delta) \leq\exp\{-N(V_{\bar{O}}(z^{*},y)-V_{\widetilde{\partial O}}-\eta/3)\}+\exp\{-N(V_{\bar{O},K_{0}^{c}}(z,y)-2\eta/3)\}\\
    &  \leq \exp\{-N(V_{\bar{O},K_{0}^{c}}(z,y)\wedge[V_{\bar{O}}(z^{*},y)-V_{\widetilde{\partial O}}]-\eta)\}.
\end{align*}
We now show that 
\begin{equation}\label{equal}
V_{\bar{O},K_{0}^{c}}(z,y)\wedge[V_{\bar{O}}(z^{*},y)-V_{\widetilde{\partial O}}]=V_{\bar{O}}(z,y)\wedge[V_{\bar{O}}(z^{*},y)-V_{\widetilde{\partial O}}]
\end{equation}
Indeed we first remark that $V_{\bar{O},K_{0}^{c}}(z,y)\geq V_{\bar{O}}(z,y)$. If $V_{\bar{O},K_{0}^{c}}(z,y)> V_{\bar{O}}(z,y)$, then the nearly minimal paths for 
$V_{\bar{O}}(z,y)$ must come arbitrarily close to $K_{0}$, so
\begin{equation*}
V_{\bar{O}}(z,y)=V_{\bar{O}}(z,z^{*})+V_{\bar{O}}(z^{*},y)=V_{\bar{O}}(z^{*},y),
\end{equation*}
and then $V_{\bar{O}}(z^{*},y)-V_{\widetilde{\partial O}}<V_{\bar{O}}(z,y)<V_{\bar{O},K_{0}^{c}}(z,y)$ and \eqref{equal} is true. This establishes the upper bound \eqref{upper13}.
\par In order to obtain the lower bound \eqref{lower13}, we deduce from the lower bounds for
 $\Pi_{1}$ and $\Pi_{2}$ that
\begin{align*}
      \mathbb{P}_{z}(|\tilde{Z}^{N}(\tau^{N}_{O})-y|<\delta) &\geq \exp\{-N(V_{\bar{O}}(z^{*},y)-V_{\widetilde{\partial O}}+\eta)\}+\exp\{-N(V_{\bar{O},K_{0}^{c}}(z,y)+\eta)\}\\
    &  \geq\exp\{-N(V_{\bar{O},K_{0}^{c}}(z,y)\wedge[V_{\bar{O}}(z^{*},y)-V_{\widetilde{\partial O}}]+\eta)\},
\end{align*}
and then \eqref{lower13} follows from \eqref{equal}.
\end{description}
\end{proof}

\begin{corollary}\label{corunicite}
 Assume that there exits a unique point $y^{*}\in\widetilde{\partial O}$  such that 
 \begin{equation*} 
 V_{\bar{O}}(z^{*},y^{*})=V_{\widetilde{\partial O}}=\inf_{y\in\widetilde{\partial O}}V_{\bar{O}}(z^{*},y),
 \end{equation*}  
 then for all $\delta>0$, and  $z\in O$
\begin{equation*}
\lim_{N\to\infty}\mathbb{P}_{z}(|\tilde{Z}^{N}(\tau^{N}_{O})-y^{*}|<\delta)=1.
\end{equation*}
\end{corollary}

\section{Application to four models}\label{sec:appl}
In the four following models , we show that for $N$ large, with high probability the stochastic process hits the boundary of the basin or attraction of the endemic equilibrium near the point to which the law of large number limit converges, when restricted to this boundary. Note that $V_{\widetilde{\partial O}}$ is the value function of a deterministic optimal control problem. We will exploit  Pontryaguin's maximum principle, in order to prove the results of this section. 
Let us first describe Pontryaguin's maximum principle in the case of the optimal control problems we are concerned with here, then we shall present the four models successively, and finally we shall prove our result.
Note that Assumption \ref{assump2} is easily verified in the first two examples, and has been carefully verified in \cite{Pard2017} for the last two.

\subsection{Pontryaguin maximum principle}
Let us formulate the optimal control problem, of which  $V_{\widetilde{\partial O}}$ is the value function. Let
\[ \dot{z}_t=Bu_t,\quad z_0=z^\ast,\]
where $B$ is the $d\times k$ matrix whose $j$--th column is the vector $h_j$, $1\le j\le k$, and the control $u\in L^1([0,\infty);\R^k_+)$ is to be chosen together with the final time $T$ such as to minimize the cost functional
\[ C(u)=\sum_{j=1}^k\int_0^T f(u_j(t),\beta_j(z_t)) dt,\quad \text{where }f(a,b)=a\log(a/b)-a+b,\]
subject to the constraint : $z_T\in M$, where $M$ is the set of points $z\in A$ (or $\in A_R$ in the case of the $SIR$ model with demography) which are such that 
\[\text{if } \dot{z}_t=b(z_t)=\sum_{j=1}^k \beta_j(z)h_j,\ z_0=z,\text{ then }z_t\to\bar z,\ \text{ as }t\to\infty.\]

 We associate to the above control problem the Hamiltonian
\[ H(z,r,u)=\langle r,Bu\rangle-\sum_{j=1}^k f(u_j,\beta_j(z)).\]
The maximum principle states that (see \cite{PBGM} and \cite{trelat2008controle})
\begin{proposition}\label{pont}
If $(\hat{T}; \hat{u}_t, 0\le t\le \hat{T})$ is an optimal pair, then there exists an adjoint state $r\in C([0,\hat{T}];\R^d)$ such that
\begin{align*}
 \dot{z}_t&=B\hat{u}_t,: z_0=z^\ast,\ z_{\hat{T}}\in M,\\
\dot{r}_t&=\sum_{j=1}^k\left[\nabla\beta_j(z_t)-\hat{u}_t^j\frac{\nabla\beta_j(z_t)}{\beta_j(z_t)}\right],\ r_{\hat{T}}\perp M ,\\
H(z_t,r_t,\hat{u}_t)&=\max_{v\in \R^k_+}H(z_t,r_t,v)=0,\ 0\le t\le \hat{T}.
\end{align*}
\end{proposition}
Note that $r_{\hat{T}}\perp M$ means that $r_{\hat{T}}$ and the tangent vector to $M$ at $z_{\hat{T}}$ are perpendicular.
 The fact that the Hamiltonian is zero at the final time is a consequence of two facts:
the final time $T$ is not fixed and there is no final cost; the fact that the Hamiltonian is zero along the optimal trajectory then follows from the fact that it is constant, since none of the coefficient depends upon the variable $t$. 

Let us exploit Proposition \ref{pont} to rewrite the system of ODEs for $z_t$ and $r_t$ along an optimal trajectory.
We denote by $B^\ast$ the transposed of the matrix $B$.
Since for each $1\le j\le k$, $u\to (B^\ast r)_ju-f(u,\beta_j(z))$ is concave, its maximum is achieved at the zero of its derivative, if it is non negative. We conclude that $\hat{u}_j=e^{(B^\ast r)_j}\beta_j(z)$. Consequently, along the optimal trajectory,
\begin{equation}\label{fbode}
\dot{z}_t=\sum_{j=1}^k e^{(B^\ast r_t)_j}\beta_j(z_t)h_j,\quad \dot{r}_t=\sum_{j=1}^k\left(1-e^{(B^\ast r_t)_j}\right)\nabla\beta_j(z_t).\end{equation}
Moreover, the optimal cost reads
\[ \hat{C}=\sum_{j=1}^k\int_0^{\hat{T}}\left(1-e^{(B^\ast r_t)_j}+(B^\ast r_t)_je^{(B^\ast r_t)_j}\right)\beta_j(z_t)dt.\]

In the examples below, $d=2$. We shall denote by $x_t$ and $y_t$ the two components of $z_t$, and by $p_t$ and $q_t$ the two components of $r_t$.

\subsection{The $SIRS$ model}
Let $x_t$ denote the proportion of infectious individuals in the population and $y_t$ the proportion of susceptibles. Since the total population size is constant, $1-x_t-y_t$ is the proportion of removed (also called recovered) individuals, who lose their immunity and become susceptible again at rate $\rho$. 
The deterministic SIRS model, see e.g. \cite{BP}, can be written as
\begin{align*}
\dot{x}_t&=\lambda x_t y_t-\gamma x_t,\\
\dot{y}_t&=-\lambda x_t y_t+\rho(1-x_t-y_t).
\end{align*}
If the basic reproduction number $R_0=\lambda/\gamma>1$, then there is a unique stable endemic equilibrium
$z^\ast=\left(\frac{\rho}{\lambda}\frac{\lambda-\gamma}{\rho+\gamma},\frac{\gamma}{\lambda}\right)$. The disease free equilibrium is $\bar z=(0,1)$. Here $\widetilde{\partial O}=\{x=0\}$.

The corresponding SDE is of the form \eqref{EqPoisson1}, with $d=2$, $k=3$,
$h_1=\begin{pmatrix}1\\ -1\end{pmatrix}$, $h_2=\begin{pmatrix}-1\\ 0\end{pmatrix}$, $h_3=\begin{pmatrix}0\\ 1\end{pmatrix}$,
$\beta_1(x,y)=\lambda xy$, $\beta_2(x,y)=\gamma x$, $\beta_3(x,y)=\rho(1-x-y)$. 

Note that in this example and in the next one, we do not need the definition of the reflected process $\tilde{Z}^N_t$, since it is identical to $Z^N_t$.

The system of ODES for the state and adjoint state \eqref{fbode} reads in this case
\begin{align*}
\dot{x}_t&=\lambda e^{p_t-q_t} x_ty_t-\gamma e^{-p_t}x_t,\\
\dot{y}_t&=-\lambda e^{p_t-q_t} x_ty_t+\rho e^{q_t}(1-x_t-y_t),\\
\dot{p}_t&=\lambda (1-e^{p_t-q_t}) y_t+\gamma(1-e^{-p_t})-\rho(1-e^{q_t}),\\
\dot{q}_t&=\lambda (1-e^{p_t-q_t}) x_t-\rho(1-e^{q_t}),\\
x_0&=\frac{\rho}{\lambda}\frac{\lambda-\gamma}{\rho+\gamma},\ y_0=\frac{\gamma}{\lambda},\ x_T=0,\ q_T=0.
\end{align*}

\subsection{The $SIR$ model with demography}
The deterministic SIR model with demography, see e.g. \cite{BP}, can be written as
\begin{align*}
\dot{x}_t&=\lambda x_t y_t-(\gamma+\mu) x_t,\\
\dot{y}_t&=-\lambda x_t y_t+\mu-\mu y_t.
\end{align*}
Again $x_t$ (resp. $y_t$) denotes the proportion of infectious (resp. susceptible) individuals in the population.
As opposed to the SIRS model, the removed individuals do not loose their immunity, rather new susceptibles are born at rate
$\mu$, which is the rate at which both susceptibles and infectious die (the infectious heal at rate $\rho$ as in the $SIRS$ model). 
If $R_0=\frac{\lambda}{\gamma+\mu}>1$, there is a stable endemic equilibrium $z^\ast=\left(
\frac{\mu}{\gamma+\mu}-\frac{\mu}{\lambda},\frac{\gamma+\mu}{\lambda} \right)$, and a disease free equilibrium $\bar z=(0,1)$. Here $\widetilde{\partial O}=\{x=0\}$.

The corresponding SDE is of the form \eqref{EqPoisson1}, with $d=2$, $k=4$, $h_1=\begin{pmatrix}1\\ -1\end{pmatrix}$, $h_2=\begin{pmatrix}-1\\ 0\end{pmatrix}$, $h_3=\begin{pmatrix}0\\ 1\end{pmatrix}$, $h_4=\begin{pmatrix}0\\-1\end{pmatrix}$,
$\beta_1(x,y)=\lambda xy$, $\beta_2(x,y)=(\gamma+\mu) x$, $\beta_3(x,y)=\mu$, $\beta_4(x,y)=\mu y$.
One difficulty in this model is that the ``proportions'' here are not true proportions, they can be greater than 1. In fact the random process lives in all of $\R^2_+$. However one can show, see \cite{BP} that the cost needed to hit the boundary $\{x+y=R\}$ tends to $\infty$ as $R\to\infty$, hence for $R$ large enough,  if we restrict ourself to the subset $A_R=\{(x,y)\in\R^2_+,\ x+y\le R\}$,
since
$\min_{z\in\partial A_R}V(z^\ast,z)=\min_{z_1=0}V(z^\ast,z)$. Also $A_R$ is not exactly $A_1$ which has been considered so far, it is easily seen that all our results extend to this new situation.
The system of ODES for the state and adjoint state \eqref{fbode} reads in this case
\begin{align*}
\dot{x}_t&=\lambda e^{p_t-q_t} x_ty_t-(\gamma+\mu) e^{-p_t}x_t,\\
\dot{y}_t&=-\lambda e^{p_t-q_t} x_ty_t+\mu e^{q_t}-\mu e^{-q_t}y_t,\\
\dot{p}_t&=\lambda (1-e^{p_t-q_t}) y_t+(\gamma+\mu)(1-e^{-p_t}),\\
\dot{q}_t&=\lambda (1-e^{p_t-q_t}) x_t+\mu(1-e^{-q_t}),\\
x_0&=\frac{\mu}{\gamma+\mu}-\frac{\mu}{\lambda},\ y_0=\frac{\gamma+\mu}{\lambda},\ x_T=0,\ q_T=0.
\end{align*}
\subsection{The $SIV$ model}
In this model, some of the individuals are vaccinated. Also the vaccine is not perfect, it gives a partial protection. If we denote by 
$x_t$ (resp. $y_t$) the proportion of infectious  (resp. of vaccinated) individuals at time $t$, the model studied by \cite{kribs2000simple} reads
\begin{align*}
\dot{x}_t&=(\beta -\mu-\gamma)x_t -\beta(1-\chi)x_ty_t-\beta x_t^2,\\
\dot{y}_t&=\eta(1-x_t)-(\eta+\mu+\theta)y_t-\chi\beta x_ty_t.
\end{align*}
For certain values of the parameters, it is shown that this model have one disease--free equilibrium, one locally stable endemic equilibrium $z^\ast$, and a third equilibrium $\bar z$ which lies on the characteristic boundary which separates the basins of attraction of the two other equilibria, which is here $\widetilde{\partial O}$. 
The corresponding SDE is of the form \eqref{EqPoisson1}, with $d=2$, $k=7$, $h_1= \begin{pmatrix}1\\ 0\end{pmatrix}$,
$h_2=\begin{pmatrix}1\\ -1\end{pmatrix}$, $h_3=\begin{pmatrix}-1\\ 0\end{pmatrix}$, $h_4=\begin{pmatrix}0\\ -1\end{pmatrix}$, $h_5=\begin{pmatrix}0\\ 1\end{pmatrix}$, $h_6=\begin{pmatrix}-1\\ 0\end{pmatrix}$ and $h_7=\begin{pmatrix}0\\ -1\end{pmatrix}$, $\beta_1(x,y)=\beta x(1-x-y)$, $\beta_2(x,y)=\chi\beta xy$, $\beta_3(x,y)=\gamma x$, $\beta_4(x,y)=\theta y$, 
$\beta_5(x,y)=\eta(1-x-y)$, $\beta_6(x,y)=\mu x$, $\beta_7(x,y)=\mu y$.

The system of ODES for the state and adjoint state \eqref{fbode} reads in this case
\begin{align*}
\dot{x}_t&=\beta e^{p_t} x_t(1-x_t-y_t)+\chi\beta e^{p_t-q_t} x_ty_t-(\gamma+\mu) e^{-p_t}x_t,\\
\dot{y}_t&=-\chi\beta e^{p_t-q_t} x_ty_t-(\theta+\mu) e^{-q_t}y_t+\eta e^{q_t}(1-x_t-y_t),\\
\dot{p}_t&=\beta (1-e^{p_t})(1-2x_t- y_t)+\chi\beta(1-e^{p_t-q_t})y_t+\gamma(1- e^{-p_t})-\eta(1-e^{q_t})+\mu(1-e^{-p_t}),\\
\dot{q}_t&=-\beta(1-e^{p_t})x_t+\chi\beta (1-e^{p_t-q_t}) x_t+\theta(1-e^{-q_t})+\eta(e^{q_t}-1)+\mu(1-e^{-q_t}),\\
x_0&=x^\ast,\ y_0=y^\ast,\ (x_T,y_T)\in M,\ (p_T,q_T)\perp M.
\end{align*}

\subsection{The $S_0IS_1$ model}
 This is a version of the $SIR$ model, where the recovered individuals are susceptible, but with a susceptibility which is less that that of those who have never been infected. They are of type $S_1$. This model has been studied in \cite{safan2006minimum}. Let $x_t$ (resp. $y_t$) denote the proportion of infectious  (resp. of type $S_1$) individuals. The ODE reads
 \begin{align*}
\dot{x}_t&=\beta(1-x_t -y_t)x_t -(\mu+\alpha) x_t+r\beta x_ty_t,\\
\dot{y}_t&=\alpha  x_t-\mu y_t-r\beta x_ty_t.
\end{align*}
Again for certain values of the parameters, we have the same large time description as for the $SIV$ model, and $\widetilde{\partial O}$ is the characteristic boundary which separates the basins of attraction of the two local stable equilibria.
The corresponding SDE is of the form \eqref{EqPoisson1}, with $d=2$, $k=5$, $h_1= \begin{pmatrix}1\\ 0\end{pmatrix}$,
$h_2=\begin{pmatrix}-1\\ 1\end{pmatrix}$, $h_3=\begin{pmatrix}-1\\ 0\end{pmatrix}$, $h_4=\begin{pmatrix}1\\ -1\end{pmatrix}$ and $h_5=\begin{pmatrix}0\\ -1\end{pmatrix}$, $\beta_1(x,y)=\beta x(1-x-y)$, $\beta_2(x,y)=\alpha x$, $\beta_3(x,y)=\mu x$, $\beta_4(x,y)=r\beta xy$ and $\beta_5(x,y)=\mu y$.

The system of ODES for the state and adjoint state \eqref{fbode} reads in this case
\begin{align*}
\dot{x}_t&=\beta e^{p_t} x_t(1-x_t-y_t)-\alpha e^{q_t-p_t} x_t-\mu e^{-p_t}x_t+r\beta e^{p_t-q_t}x_ty_t,\\
\dot{y}_t&=\alpha e^{q_t-p_t} x_t-r\beta e^{p_t-q_t}x_ty_t-\mu e^{-q_t}y_t,\\
\dot{p}_t&=\beta (1-e^{p_t})(1-2x_t- y_t)+\alpha(1-e^{q_t-p_t})+\mu(1- e^{-p_t})+r\beta(1-e^{p_t-q_t})y_t,\\
\dot{q}_t&=-\beta(1-e^{p_t})x_t+r\beta(1-e^{p_t-q_t})x_t+\mu(1-e^{-q_t}),\\
x_0&=x^\ast,\ y_0=y^\ast,\ (x_T,y_T)\in M,\ (p_T,q_T)\perp M.
\end{align*}

\subsection{The result}
In the four above examples, extinction happens when $z_t$ hits the boundary to which $\bar z$ belongs (which is what is denoted $\widetilde{\partial O}$ in the previous sections of this paper). This is quite clear in the first two examples. In the last two, as soon as the process crosses that boundary, it converges very quickly (in zero time in the scale of Large Deviations) to the disease free equilibrium. When on the boundary $\widetilde{\partial O}$, the solution of the ODE converges to $\bar z$. It clearly follows from that remark that
$V_{\widetilde{\partial O}}=V_{\overline{O}}(z^\ast,\bar z)$. What we want to show is that this minimum is unique, i.e. 
\begin{proposition}\label{pro:minV}
In each of the above four examples, for all $z\in\widetilde{\partial O}\backslash\{\bar z\}$, $V_{\overline{O}}(z^\ast,\bar z)<V_{\overline{O}}(z^\ast, z)$. 
\end{proposition}
This shows, thanks to
Corollary \ref{corunicite}, that for large $N$, $Z^{N,z}$ will exit the domain of attraction of the endemic equilibrium $z^\ast$ in the vicinity of $\bar z$ with probability almost $1$.

We now turn to the

\noindent{\it Proof of Proposition \ref{pro:minV}}
Suppose first that there exists a minimizing sequence $\{u_n,\ n\ge1\}\subset L^1(\R_+;\R^k_+)$ such that the corresponding trajectory $z_t$ hits the target $M$ at some point $z_n\in\widetilde{\partial O}\backslash\{\bar z\}$ in time $T_n$, with $T:=\sup_n T_n<\infty$. 
Since $I_T$ is a good rate function (i.e. its level sets are compact), see Theorem 3 in \cite{Sam2017}, there exists a subsequence $z_{n_k}$ which converges to a  optimal trajectory $\hat{z}$ which hits $M$ at point 
$z\not=\bar z$ at time $\hat{T}\le T$, with a control 
$\hat{u}\in L^1([0,T];\R^k_+)$. We concatenate $\hat{z}$ with the solution of the ODE starting from $z$, which converges to $\bar z$ (in infinite time). Since the second part of the trajectory runs at no cost, the whole trajectory is optimal for the same control problem as above, but with the constraint that $\bar z$ must be the final point. We apply the Pontryagin maximum principle to this new optimal control problem, which implies the existence of a continuous adjoint state $(p_t,q_t)$. Since $p_t=q_t=0$ for $t>\hat{T}$, we have $p_{\hat{T}}=q_{\hat{T}}=0$. But this is not possible. $z_t$ being bounded, the solution $(p_t,q_t)$ of the adjoint state equation cannot hit $(0,0)$ in finite time. One way to see this is to note that the function $(p_t,q_t)$ time reversed from time $\hat{T}$ would solve an ODE starting from $(0,0)$, whose unique solution is $(p_t,q_t)\equiv(0,0)$, see the second equation in \eqref{fbode}. 
We conclude from the above argument  that if an optimal trajectory converges to some point $z\not=\bar z$, then it does so in infinite time. 
Consequently $(x_\infty,y_\infty,p_\infty,q_\infty)$ must be a fixed point of \eqref{fbode}. It remains to show that 
$(\bar x,\bar y,0,0)$ is the only admissible fixed point. The argument is now slightly different in the various 
considered examples.

In the first two examples ($SIRS$ and $SIR$ with demography), we know that $x_\infty$ (the first coordinate of $z_\infty$) and the second coordinate $q_\infty$ of $r_\infty$ vanish. Consequently 
$y_\infty$ must be the zero of $1-y$, hence equals $1$.

In the two other cases,  we first note that both $p_\infty$ and $q_\infty$ must be finite. Indeed, either 
the solution must remain bounded, or else would explode in finite time, which contradicts the existence of the adjoint state on $[0,+\infty)$. We next show that both coordinates of $z_\infty$, $x_\infty$ and $y_\infty$, are positive. In case of the $SIV$ model, we first note that $y_\infty=0$ would imply $x_\infty =1$, but $(1,0)$ is clearly not on $M$. On the other hand, 
$x_\infty=0$ would imply that the second coordinate $q_\infty$ of $r_\infty$ vanishes, and $y_\infty=\frac{\eta}{\eta+\mu+\theta}$, which again gives a point not on $M$. In case of the $S_0IS_1$ model, we note that $x_\infty=0$ implies $y_\infty=0$, and the reserve implication is also true, but $(0,0)$ is not on $M$. Finally, since $\hat{C}<\infty$, the running cost must converge to $0$ as $t\to\infty$. For each $1\le j\le k$ such that $\beta_j(z_\infty)>0$, this implies that $\langle r_t,h_j\rangle\to0$ as $t\to\infty$. 
This is true for $j=3$ and $4$ in case of the $SIV$ model, for $j=3$ and $5$ in case of the $S_0IS_1$ model. In both cases, it implies that $(p_\infty,q_\infty)=(0,0)$. Consequently $z_\infty$ is a zero of 
$b(z)=\sum_{j=1}^k \beta_j(z)h_j$ and belongs to $M$, hence equals $\bar z$. The proof is complete.
\hfill $\square$


\frenchspacing
\bibliographystyle{plain}

\end{document}